\documentclass[a4paper]{article}

\usepackage[english]{babel}
\usepackage[utf8]{inputenc}
\usepackage{amsmath,amsthm,aliascnt}
\usepackage{graphicx}
\usepackage{caption}
\usepackage{subcaption}
\usepackage{nameref}
\usepackage[colorinlistoftodos]{todonotes}
\usepackage{amssymb}
\usepackage{xcolor}
\usepackage[sc]{titlesec}
\usepackage{anysize}
\usepackage{authblk}

\usepackage{hyperref}
\definecolor{lgreen}{rgb}{0.0, 0.48, 0.0}
\definecolor{lpurple}{rgb}{0.48, 0.0, 0.48}
\hypersetup{linktocpage,
            colorlinks=true,
            linkcolor=lgreen,
            citecolor=lpurple,
            linktoc=true}

\title{Shifting the Phase Transition Threshold for Random Graphs using Set Degree
Constraints\footnote{This work is partially supported by the 
 French project MetACOnc, ANR-15-CE40-0014
 and by the French project CNRS-PICS-22479.}}

\author[1,2,3]{Sergey Dovgal}
\author[2]{Vlady Ravelomanana}
\affil[1]{LIPN -- UMR CNRS 7030. Université Paris 13, 99 avenue Jean-Baptiste Clément, 93430 Villetaneuse. France\\
  \texttt{dovgal@lipn.univ-paris13.fr}}
\affil[2]{IRIF -- UMR CNRS 8243. Université Paris 7, 8 place Aurélie Nemours, 75013 Paris. France\\
  \texttt{vlad@irif.fr}}
\affil[3]{Moscow Institute of Physics and Technology, Institutskiy per. 9, Dolgoprudny, Russia 141700
}

\usepackage{xcolor}
\definecolor{bblue}{rgb}{0.2, 0.4, 0.8}
\definecolor{bgreen}{rgb}{0.2, 0.6, 0.4}
\definecolor{bred}{rgb}{0.8, 0.35, 0.2}
\definecolor{bviolet}{rgb}{0.7, 0.2, 0.7}
\definecolor{blackred}{rgb}{0.6, 0.3, 0.3}
\definecolor{blackblue}{rgb}{0.3, 0.3, 0.6}
\definecolor{byellow}{rgb}{0.7, 0.45, 0.3}

\usepackage{tikz}
\usetikzlibrary{fit,arrows,trees,shapes,snakes,shapes.geometric,calc}
\tikzset{
  treenode/.style = {align=center, inner sep=0pt, text centered,
    font=\sffamily},
  arn_rr/.style = {treenode, circle, bred, draw=bred, 
    fill=bred!10,
    minimum width=0.9em, minimum height=0.9em
  },% arbre rouge noir, noeud rouge
  arn_nn/.style = {treenode, circle, bblue, draw=bblue, 
    fill=bblue!10,
    minimum width=0.5em, minimum height=0.5em
  },% arbre rouge noir, noeud rouge
  arn_n/.style = {treenode, circle, bblue, draw=bblue, 
    text width=1.5em, very thick,
    fill=bblue!10},% arbre rouge noir, noeud rouge
  arn_v/.style = {treenode, circle, bviolet, draw=bviolet, 
    text width=1.5em, very thick,
    fill=bviolet!10},% arbre rouge noir, noeud rouge
  arn_g/.style = {treenode, circle, bgreen, draw=bgreen, 
    text width=1.5em, very thick,
    fill=bblue!10},% arbre rouge noir, noeud rouge
  arn_y/.style = {treenode, circle, byellow, draw=byellow, 
    text width=1.5em, very thick,
    fill=byellow!10},% arbre rouge noir, noeud rouge
  arn_r/.style = {treenode, circle, bred, draw=bred, 
    text width=1.5em, very thick,
    fill=red!10},% arbre rouge noir, noeud rouge
  arn_x/.style = {treenode, triangle, draw=black,
    minimum width=0.5em, minimum height=0.5em},% arbre rouge noir, nil
  triangle/.style = {treenode, bred, draw=bred, fill=bred!20, regular polygon, regular polygon
    sides=3, very thick, text width=1.5em },
  vtriangle/.style = {treenode, bviolet, draw=bviolet, fill=bviolet!20, regular polygon, regular polygon
    sides=3, very thick, text width=1.5em },
  btriangle/.style = {treenode, bblue, draw=bblue, fill=bblue!20, regular polygon, regular polygon
    sides=3, very thick, text width=1.5em },
  bpentagon/.style = {treenode, bblue, draw=bblue, fill=bblue!20, regular polygon, regular polygon
    sides=5, very thick, text width=5.0em }
}

\def \SET{\mathrm{SET}}

\newcommand{\mynewtheorem}[2]{
  \newaliascnt{#1}{dummy}
  \newtheorem{#1}[#1]{#2}
  \aliascntresetthe{#1}
  \expandafter\def\csname #1autorefname\endcsname{#2}
}

\theoremstyle{definition}

\mynewtheorem{theorem}{Theorem}
\mynewtheorem{lemma}{Lemma}
\mynewtheorem{remark}{Remark}
\mynewtheorem{corollary}{Corollary}

\begin{document}

\maketitle

\begin{abstract}
We show that by restricting the degrees of the vertices of a graph to an 
arbitrary set \( \Delta \), the threshold
point \( \alpha(\Delta) \) of the phase transition for a random graph with \( n \) vertices
and \( m = \alpha(\Delta) n \) edges can be either 
accelerated (e.g., \( \alpha(\Delta) \approx 0.381 \) for \( \Delta = \{0,1,4,5\} \)) or 
postponed (e.g., \( \alpha(\{ 2^0, 2^1, \cdots, 2^k, \cdots \}) \approx 0.795
\))
compared to a classical Erd\H{o}s--R\'{e}nyi random graph with \( \alpha(\mathbb
Z_{\geq 0}) = \tfrac12 \).
In particular, we prove that the probability of graph being nonplanar and the
probability of having a complex component, goes from \( 0 \) to \( 1 \) as \(
m \) passes \( \alpha(\Delta) n \).
We investigate these probabilities and also different graph
statistics inside the critical window of transition
(diameter, longest path and circumference of a complex component).

\end{abstract}

\section{Introduction}

\subsection{Shifting the Phase Transition}
Consider a random Erd\H{o}s--Rényi graph \( G(n,m) \)~\cite{erdos_renyi},
 that is  a graph chosen uniformly
at random among all simple graphs built with \( n \) vertices
labeled with distinct numbers from \( \{1,2,\ldots,n\} \), and \( m
\) edges. 
The range \( m =  \tfrac12 n(1 + \mu n^{-1/3}) \)
where \( n \to \infty \), and \( \mu \) depends on \( n \),
 is of particular interest since there are three distinct regimes,  according to how the crucial parameter $\mu$
 grows as $n$ is large:
 \begin{itemize}
     \item as \( \mu \to -\infty \),  the size of the largest component
is of order \( \Theta(\log n) \), and the connected components
are almost surely trees and unicyclic components;
    \item next, inside what is known as the critical window 
\( |\mu| = O(1) \), the largest
component size is of order \( \Theta(n^{2/3}) \) and complex structures
(unempty set of connected components having strictly more edges
than vertices) start to appear with significant probabilities;
    \item 
finally, as \( \mu \to +\infty \) with $n$, there is typically a
unique component of size \( \Theta(n) \) called the giant component.
\end{itemize}
Since the article of Erd\H{o}s and R\'{e}nyi~\cite{erdos_renyi},
various researchers have studied
in depth the phase transition of the Erd\H os--R\'enyi random graph model
culminating with the masterful work of Janson, Knuth, \L uczak, and
Pittel~\cite{giant_component} who used enumerative approach
to analyze the fine structure of the components
inside the critical window of $G(n,\,m)$.

The last decades have seen a growth of interest in delaying or advancing
the phase transitions of random graphs.
Mainly, two kinds of processes have
been introduced and studied:
\begin{itemize}
\item[a)] the \textit{Achlioptas process} where
  models of random graph are obtained by adding edge one by one but
  according to a given rule which
  allows to choose the next edge from a set of candidate
  edges~\cite{avoiding_giant_component, riordan2017phase},
\item[b)] the \textit{given degree sequence models}
  where a sequence $(d_1,\, \cdots,\, d_n)$ of degrees is given and
  a simple graph built on $n$ vertices is uniformly chosen
  from the set of all graphs whose degrees match with the sequence $d_i$
  (see ~\cite{Riordan12,molloy,degree_sequence_giant_components,fixed_degree_sequence_giant_component}).
\end{itemize}

In \cite{avoiding_giant_component,achlioptas_transition,riordan2017phase},
the authors studied the Achlioptas process.
Bohman and Frieze~\cite{avoiding_giant_component}
were able to show
that there is a random graph process such that after
adding \( m = 0.535 \, n > 0.5 \, n \) edges the size of the largest
component is (still) polylogarithmic in \( n \) which contrasts with
the classical Erd\H os--R\'enyi random graphs.
Initially, this process was conjectured to have a different local geometry of
transition compared to classical Erd\H{o}s--R\'{e}nyi model, but
Riordan and Warnke~\cite{achlioptas_transition,riordan2017phase} were able to show that
this is not the case.
Next, in the model of random graphs with a fixed 
degree sequence $D=(d_1,\,\cdots,\,d_n)$,
Joos, Perarnau, Rautenbach,
and Reed~\cite{fixed_degree_sequence_giant_component} proved that 
a simple condition that a graph with degree sequence $D$ has
a connected component of linear size, is that the sum
of the degrees in $D$ which are not $2$ is at least
$\lambda(n)$ for some function $\lambda(n)$ that goes to
infinity with $n$. 

In the current work, our approach is rather different. We study
\textit{random graphs with degree constraints} that are graphs 
drawn uniformly at random from the set of all graphs
with given number of vertices and edges
with all vertices having degrees from a given set $\Delta \subseteq \mathbb{Z}_{\geq
0}$, with the only restriction \( 1 \in \Delta \), which we discuss below.
De~Panafieu and Ramos~\cite{degree_constraints} calculated asymptotic number of
such graphs using methods from analytic combinatorics.
Using their asymptotic results,
we prove that random graphs with degrees from
the set \( \Delta \) have their phase
transition shifted from the density of edges \(\tfrac{m}{n}= \tfrac12 \) to
\(\tfrac{m}{n}= \alpha \) for an \textit{explicit} and
\textit{computable} constant \( \alpha = \alpha(\Delta) \)
and  the new critical window of transition becomes
\(
    m = \alpha n (1 \pm \mu n^{-1/3})\, .
\)

In addition, we also prove that
the structure of such graphs inside this crucial window
behaves as in the Erd\H{o}s--R\'{e}nyi case. For instance, we
prove that extremal parameters such as the diameter, the circumference
or the longest path are of order $\Theta(n^{1/3})$ around
$m=\alpha n$. The size of complex components of our graphs are of
order $\Theta(n^{2/3})$ as $\mu$ is bounded.
A very similar result but about the diameter
of the largest component of $G(n,\, p=\frac{1}{n}+\frac{\mu}{n^{4/3}})$
has been obtained by Nachmias and Peres~\cite{mixing_time}
(using very different methods).

In the seminal paper of Erd\H os and R\'enyi,
amongst other non-trivial properties, they
discussed  the planarity of 
random graphs with various edge densities~\cite{erdos_renyi}.
The probabilities of planarity of Erd\H{o}s-R\'enyi random graphs
inside their window of transition have been since then computed by Noy,
Ravelomanana, and Ru\'{e}~\cite{critical_planarity}.
In the current work, we extend this study by
showing that the planarity threshold shifts from $\frac{n}{2}$ for classical
random graphs to    $\alpha n$ for graphs with degrees
from $\Delta$. More precisely,
first we show that such objects are almost surely
planar as $\mu$ goes to \( - \infty \) and non-planar
as  $\mu$ tends to $+\infty$.  Next, 
as function of $\mu$, we compute
the limiting probability that random graphs of degrees in
$\Delta$ are planar as $\mu = O(1)$.

Our work is motivated by the following research questions:
(i) what can be the contributions of analytic combinatorics to study constrained
random graphs?
(ii) the birth of the giant component often corresponds to drastic changes in
the complexities of several algorithmic optimization / decision problems on
random graphs, so by tuning the thresholds one can shift the location of hard
random instances.

\subsection{Preliminaries}

The \textit{excess} of a connected graph is
the number of its edges minus
the number of its vertices. 
For example, connected graphs with excess \( -1 \) are trees, with excess \( 0 
\)~--- graphs with one 
cycle (also known as unicycles or unicyclic graphs),
connected bicycles have excess \( 2 \), and so on (see \autoref{fig:different:excess}).
Connected graph always has excess at least \( -1 \). A connected component with excess at
least $1$, is called a \emph{complex component}.
The \emph{complex part} of a random graph is the union of its complex components.

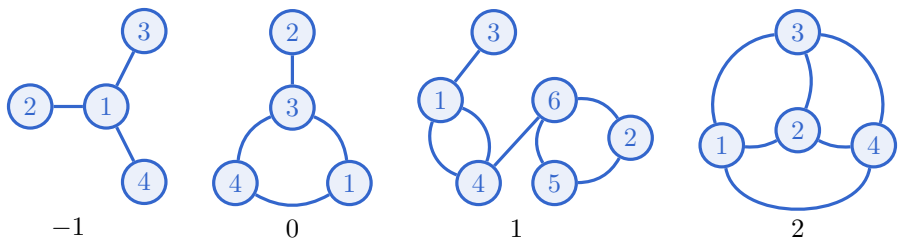
\begin{figure}[hbt]
\centering
\begin{tikzpicture}[>=stealth',thick, node distance=1.0cm] 
\draw
node[arn_n](a) at (0.0, 0.3) {\( 2 \)}
node[arn_n](b) at (1.0, 0.3) {\( 1 \)}
node[arn_n](c) at (1.5, 1.3) {\( 3 \)}
node[arn_n](d) at (1.5, -.7) {\( 4 \)}
node at (0.5, -1.3) {\( -1 \)}
;
\path (a) edge [bblue, very thick] (b) ;
\path (b) edge [bblue, very thick] (c) ;
\path (b) edge [bblue, very thick] (d) ;
\end{tikzpicture}
\( \quad \) 
\begin{tikzpicture}[>=stealth',thick, node distance=1.0cm] 
\draw
node[arn_n](a) at (0.0, -.5) {\( 4 \)}
node[arn_n](b) at (1.5, -.5) {\( 1 \)}
node[arn_n](c) at (.75, 0.5) {\( 3 \)}
node[arn_n](d) at (.75, 1.5) {\( 2 \)}
node at (0.75, -1.1) {\( 0 \)}
;
\path (a) edge [bblue, very thick, bend right=30] (b) ;
\path (b) edge [bblue, very thick, bend right=30] (c) ;
\path (c) edge [bblue, very thick, bend right=30] (a) ;
\path (c) edge [bblue, very thick] (d) ;
\end{tikzpicture}
\( \quad \) 
\begin{tikzpicture}[>=stealth',thick, node distance=1.0cm] 
\draw
node[arn_n](a) at (0.0, 0.0) {\( 1 \)}
node[arn_n](b) at ( .7,  .9) {\( 3 \)}
node[arn_n](c) at ( .5,-1.1) {\( 4 \)}
node[arn_n](d) at (1.5, 0.0) {\( 6 \)}
node[arn_n](e) at (1.5,-1.1) {\( 5 \)}
node[arn_n](f) at (2.5, -.5) {\( 2 \)}
node at (1.0, -1.7) {\( 1 \)}
;
\path (a) edge [bblue, very thick] (b) ;
\path (a) edge [bblue, very thick, bend left  = 45] (c) ;
\path (a) edge [bblue, very thick, bend right = 45] (c) ;
\path (c) edge [bblue, very thick] (d) ;
\path (d) edge [bblue, very thick, bend right = 30] (e) ;
\path (e) edge [bblue, very thick, bend right = 30] (f) ;
\path (f) edge [bblue, very thick, bend right = 30] (d) ;
\end{tikzpicture}
\( \quad \) 
\begin{tikzpicture}[>=stealth',thick, node distance=1.0cm] 
\draw
node[arn_n](a) at (0.0, -.5) {\( 1 \)}
node[arn_n](b) at (1.0, -.3) {\( 2 \)}
node[arn_n](c) at (2.0, -.5) {\( 4 \)}
node[arn_n](d) at (1.0, 1.0) {\( 3 \)}
node at (1.0, -1.6) {\( 2 \)}
;
\path (a) edge [bblue, very thick, bend right = 15] (b) ;
\path (b) edge [bblue, very thick, bend right = 15] (c) ;
\path (b) edge [bblue, very thick, bend right = 20] (d) ;
\path (a) edge [bblue, very thick, bend left  = 50] (d) ;
\path (d) edge [bblue, very thick, bend left  = 50] (c) ;
\path (a) edge [bblue, very thick, bend right = 80] (c) ;
\end{tikzpicture}
\caption{\label{fig:different:excess}
Examples of connected labeled graphs with different excess. As a whole, can be considered as
a graph with total excess \( -1 + 0 + 1 + 2 = 2 \)
}
\end{figure}

Next, we introduce the notion of a \emph{\( 2 \)-core (the core)} and a \emph{\( 3 \)-core (the
kernel)} of a graph. The 2-core is obtained by repeatedly removing all vertices
of degree \( 1 \) (smoothing). The 3-core is obtained from 2-core
by repeatedly replacing vertices of
degree two with their adjacent edges by a single edge connecting the neighbors
of deleted vertices (we call this a \emph{reduction procedure}). A 3-core can be a
multigraph, i.e. there can be loops and multiple edges. There is only a finite number of
connected 3-cores with a given excess \cite{giant_component}. The inverse
images of vertices of 3-core under the reduction procedure, are called
\emph{corner vertices} (cf. \autoref{fig:extremal:statistics}).
A \emph{2-path} is an inverse image of an edge in a $3$-core, i.e. a path
connecting two corner vertices.

The \emph{circumference} of a graph is the length of its longest cycle.
A \emph{diameter} of a graph is the maximal length of the shortest path
taken over all distinct
pairs of vertices. It is known that the problems of finding the longest path and the
circumference are NP-hard.

%Assume that
%\(
%    ( x_1,\ldots, x_m, y_1, \ldots, y_m )
%    \in
%    \{\xi_1, \ldots, \xi_n, \n \xi_1, \ldots, \n \xi_n \}^{2m}
%\).
%Let
%\(
%    (\xi_i)_{i=1}^n
%\)
%be a collection of \( n \) Boolean variables.
%Given a 2-\textsc{cnf} formula \( \bigwedge_{i = 1}^m(x_i \lor y_i) \), its
% \emph{digraph representation} is the directed graph with the set of edges 
%\[
%    \bigcup_{i = 1}^m\Big( (\n x_i \to y_i) \cup (\n y_i \to x_i) \Big)
%    \enspace .
%\]
%If there exists an assignation of the Boolean \( \xi_i \) such that each clause
%of the formula is satisfiable, then we
%say that formula is \textsc{sat}, otherwise we say that it is \textsc{unsat}.
%Inside a directed graph,
%a \emph{cycle} is an undirected path \( (v_1,\ldots,v_k) \) such that \(v_1 = v_k\).
%A \emph{circuit} is a directed cycle \( (v_1,\ldots,v_k) \) such that \( v_1 =
%v_k \). We say that \( x \rightsquigarrow y \) if there exists a directed path
%from \( x \) to \( y \).
%See \autoref{fig:disjoint:trees} for an example of
%digraph representation.

\textit{Random graph with degree constraints} is a graph sampled
uniformly at random from the set of all possible graphs \( \mathcal G_{n,m,\Delta}
\) having \( m \) edges and $n$ vertices all of degrees
from the set \( \Delta = \{ \delta_1, \delta_2, \ldots \} \subseteq \{
0,1,2,\ldots \} \), see~\autoref{fig:random:degreeconstraint}. The set \( \Delta \) can be finite or infinite.
\textbf{In this work, we require that \( 1 \in \Delta \)}.
 This technical condition allows the existence of trees and tree-like
 structures in the random objects under consideration. We don't know what
 happens when \( 1 \notin \Delta \).

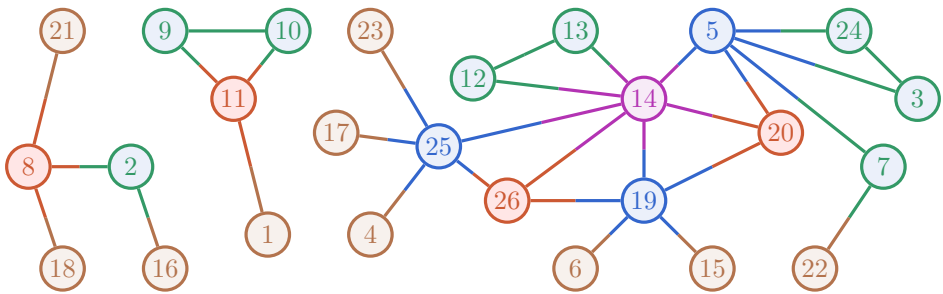
\begin{figure}[hbt]
\centering
\begin{tikzpicture}[>=stealth',thick, node distance=1.0cm, scale=0.9] 
\draw
node[arn_y](1x7) at (-8.5, -3.0) {\( 1  \)} %  1_7     
node[arn_g](2x8) at (-10.5,-2.0) {\( 2  \)} %  2_8     
node[arn_g](2x3) at ( 1.0, -1.0) {\( 3  \)} %  2_3     
node[arn_y](1x6) at (-7.0, -3.0) {\( 4  \)} %  1_6     
node[arn_n](5x1) at (-2.0,  0.0) {\( 5  \)} %  5_1     
node[arn_y](1x4) at (-4.0, -3.5) {\( 6  \)} %  1_4     
node[arn_g](2x5) at ( 0.5, -2.0) {\( 7  \)} %  2_5     
node[arn_r](3x4) at (-12.0,-2.0) {\( 8  \)} %  3_4     
node[arn_g](2x6) at (-10.0, 0.0) {\( 9  \)} %  2_6     
node[arn_g](2x7) at (-8.2,  0.0) {\( 10 \)} %  2_7     
node[arn_r](3x3) at (-9.0, -1.0) {\( 11 \)} %  3_3     
node[arn_g](2x2) at (-5.5, -0.7) {\( 12 \)} %  2_2     
node[arn_g](2x1) at (-4.0,  0.0) {\( 13 \)} %  2_1     
node[arn_v](7x1) at (-3.0, -1.0) {\( 14 \)} %  7_1     
node[arn_y](1x5) at (-2.0, -3.5) {\( 15 \)} %  1_5     
node[arn_y](1xa) at (-10.0,-3.5) {\( 16 \)} %  1_a     
node[arn_y](1x3) at (-7.5, -1.5) {\( 17 \)} %  1_3     
node[arn_y](1x9) at (-11.5,-3.5) {\( 18 \)} %  1_9     
node[arn_n](5x3) at (-3.0, -2.5) {\( 19 \)} %  5_3     
node[arn_r](3x2) at (-1.0, -1.5) {\( 20 \)} %  3_2     
node[arn_y](1x8) at (-11.5, 0.0) {\( 21 \)} %  1_8     
node[arn_y](1x1) at (-0.5, -3.5) {\( 22 \)} %  1_1     
node[arn_y](1x2) at (-7.0,  0.0) {\( 23 \)} %  1_2     
node[arn_g](2x4) at ( 0.0,  0.0) {\( 24 \)} %  2_4     
node[arn_n](5x2) at (-6.0, -1.7) {\( 25 \)} %  5_2     
node[arn_r](3x1) at (-5.0, -2.5) {\( 26 \)} %  3_1     
;
\path (5x1) edge [bblue,   very thick] ($ (5x1) !.5! (7x1) $) ;
\path (5x1) edge [bblue,   very thick] ($ (5x1) !.5! (3x2) $) ;
\path (5x1) edge [bblue,   very thick] ($ (5x1) !.5! (2x3) $) ;
\path (5x1) edge [bblue,   very thick] ($ (5x1) !.5! (2x4) $) ;
\path (5x1) edge [bblue,   very thick] ($ (5x1) !.5! (2x5) $) ;
\path (7x1) edge [bviolet, very thick] ($ (5x1) !.5! (7x1) $) ;
\path (7x1) edge [bviolet, very thick] ($ (7x1) !.5! (2x1) $) ;
\path (7x1) edge [bviolet, very thick] ($ (7x1) !.5! (2x2) $) ;
\path (7x1) edge [bviolet, very thick] ($ (7x1) !.5! (5x2) $) ;
\path (7x1) edge [bviolet, very thick] ($ (7x1) !.5! (3x1) $) ;
\path (7x1) edge [bviolet, very thick] ($ (7x1) !.5! (5x3) $) ;
\path (7x1) edge [bviolet, very thick] ($ (7x1) !.5! (3x2) $) ;
\path (2x1) edge [bgreen,  very thick] ($ (2x1) !.5! (7x1) $) ;
\path (2x2) edge [bgreen,  very thick] ($ (2x2) !.5! (7x1) $) ;
\path (2x1) edge [bgreen,  very thick] ($ (2x1) !.5! (2x2) $) ;
\path (2x2) edge [bgreen,  very thick] ($ (2x1) !.5! (2x2) $) ;
\path (3x1) edge [bred,    very thick] ($ (3x1) !.5! (7x1) $) ;
\path (3x1) edge [bred,    very thick] ($ (3x1) !.5! (5x2) $) ;
\path (3x1) edge [bred,    very thick] ($ (3x1) !.5! (5x3) $) ;
\path (5x2) edge [bblue,   very thick] ($ (5x2) !.5! (3x1) $) ;
\path (5x3) edge [bblue,   very thick] ($ (5x3) !.5! (7x1) $) ;
\path (5x2) edge [bblue,   very thick] ($ (5x2) !.5! (7x1) $) ;
\path (5x3) edge [bblue,   very thick] ($ (5x3) !.5! (3x1) $) ;
\path (5x3) edge [bblue,   very thick] ($ (5x3) !.5! (3x2) $) ;
\path (3x2) edge [bred,    very thick] ($ (3x2) !.5! (5x3) $) ;
\path (3x2) edge [bred,    very thick] ($ (3x2) !.5! (7x1) $) ;
\path (3x2) edge [bred,    very thick] ($ (3x2) !.5! (5x1) $) ;
\path (2x3) edge [bgreen,  very thick] (2x4);
\path (2x6) edge [bgreen,  very thick] (2x7);
\path (2x3) edge [bgreen,  very thick] ($ (2x3) !.5! (5x1) $) ;
\path (2x4) edge [bgreen,  very thick] ($ (2x4) !.5! (5x1) $) ;
\path (2x5) edge [bgreen,  very thick] ($ (2x5) !.5! (5x1) $) ;
\path (2x5) edge [bgreen,  very thick] ($ (2x5) !.5! (1x1) $) ;
\path (1x1) edge [byellow, very thick] ($ (1x1) !.5! (2x5) $) ;
\path (5x2) edge [bblue,   very thick] ($ (5x2) !.5! (1x2) $) ;
\path (5x2) edge [bblue,   very thick] ($ (5x2) !.5! (1x3) $) ;
\path (5x2) edge [bblue,   very thick] ($ (5x2) !.5! (1x6) $) ;
\path (5x3) edge [bblue,   very thick] ($ (5x3) !.5! (1x4) $) ;
\path (5x3) edge [bblue,   very thick] ($ (5x3) !.5! (1x5) $) ;
\path (1x2) edge [byellow, very thick] ($ (1x2) !.5! (5x2) $) ;
\path (1x3) edge [byellow, very thick] ($ (1x3) !.5! (5x2) $) ;
\path (1x4) edge [byellow, very thick] ($ (1x4) !.5! (5x3) $) ;
\path (1x5) edge [byellow, very thick] ($ (1x5) !.5! (5x3) $) ;
\path (1x6) edge [byellow, very thick] ($ (1x6) !.5! (5x2) $) ;
\path (2x6) edge [bgreen,  very thick] ($ (2x6) !.5! (3x3) $) ;
\path (2x7) edge [bgreen,  very thick] ($ (2x7) !.5! (3x3) $) ;
\path (3x3) edge [bred,    very thick] ($ (2x6) !.5! (3x3) $) ;
\path (3x3) edge [bred,    very thick] ($ (2x7) !.5! (3x3) $) ;
\path (3x3) edge [bred,    very thick] ($ (1x7) !.5! (3x3) $) ;
\path (1x7) edge [byellow, very thick] ($ (1x7) !.5! (3x3) $) ;
\path (1x8) edge [byellow, very thick] ($ (1x8) !.5! (3x4) $) ;
\path (1x9) edge [byellow, very thick] ($ (1x9) !.5! (3x4) $) ;
\path (1xa) edge [byellow, very thick] ($ (1xa) !.5! (2x8) $) ;
\path (3x4) edge [bred,    very thick] ($ (1x8) !.5! (3x4) $) ;
\path (3x4) edge [bred,    very thick] ($ (1x9) !.5! (3x4) $) ;
\path (3x4) edge [bred,    very thick] ($ (2x8) !.5! (3x4) $) ;
\path (2x8) edge [bgreen,  very thick] ($ (2x8) !.5! (3x4) $) ;
\path (2x8) edge [bgreen,  very thick] ($ (2x8) !.5! (1xa) $) ;
\end{tikzpicture}
\caption{\label{fig:random:degreeconstraint}
Random labeled graph from \( \mathcal G_{26,30,\Delta} \) with the set of degree constraints \( \Delta = \{1,2,3,5,7 \} \)
}
\end{figure}

The set \( \mathcal G_{n,m,\Delta} \) is (asymptotically) nonempty if and only 
if 
the following condition is satisfied \cite{degree_constraints}:
\begin{description}
    \item[\( (\mathcal C)\) \label{condition:C}]
    Denote \( \mathrm{gcd}(d_1 - d_2 \colon d_1,d_2 \in \Delta) \) by 
    \textit{periodicity} \( p \). Assume that the number \( m \) of edges grows 
    linearly with the number \( n \) of vertices, with \( 2m / n \) staying in 
    a fixed compact interval of \( ]\min(\Delta), \max(\Delta)[ \), and \( p \) 
    divides \( {2m-n\cdot \min(\Delta)} \). 
\end{description}

 To a given arbitrary set \( \Delta \subseteq \{0, 1, 2, \ldots \} \), we
 associate the \textit{exponential generating function} (\textsc{egf}) \( 
\omega(z) \):
\begin{equation}
	\SET_{\Delta} (z) = \omega(z) = \sum_{d \in \Delta} \dfrac{z^d}{d!} 
	\enspace .
\end{equation}
The domain of the argument \( z \) of this function
can be either considered a subset \( [0, R) \) of the real axis
or some subset of the complex plane, depending on the context.
The function \( \phi_0(z) = \frac{z 
\omega'(z)}{\omega(z)} \), which is called
the \textit{characteristic function} of \( \omega(z) \),
is non-decreasing along real axis \cite[Proposition 
IV.5]{ac}, as well as the characteristic function \( 
\phi_1(z) = \frac{z \omega''(z)}{\omega'(z)} \) of the derivative \( \omega'(z)
\).

The value of the threshold \( \alpha \), which is used in all our theorems,
is a unique solution of the system of equations
%computed through solving a system of equations
\begin{equation}\label{alpha:system}
\begin{cases}
    \phi_1(\widehat z) = 1, \\
    \phi_0(\widehat z) = 2 \alpha.
\end{cases}
\end{equation}
A unique solution $\widehat z$ of \(
\phi_1(z) = 1, z>0 \) always exists provided
that  \( 1 \in \Delta \). This solution is computable.

\paragraph{Structure of the Article.}
In \autoref{section:phase:transition} we state our main results and give proofs
which rely on technical statements from \autoref{section:tools}.
%Next, in
%\autoref{section:sat} we introduce our 2-\textsc{cnf} model and prove the lower bound
%for \textsc{sat} probability.
In \autoref{section:experiments}, we give
the results of simulations using the recursive method from~\cite{degree_constraints}. 
Section~\ref{section:tools} contains the tools from
analytic combinatorics. Then follows \autoref{section:marking} with the
method of moments and marking tools.

\section{Phase Transition for Random Graphs}
\label{section:phase:transition}

\subsection{Structure of Connected Components}
Recall that given a set \( \Delta \), its \textsc{egf} is defined as \(
\omega(z) = \sum_{d \in \Delta} z^d / d! \), and characteristic function of \(
\omega(z) \) and its derivative \( \omega'(z) \) are given by \( \phi_0(z) = z \omega'(z) /
\omega(z) \), \( \phi_1(z) = z \omega''(z) / \omega'(z) \).
\begin{theorem}
\label{theorem:critical}
Given a set \( \Delta \) with \( 1 \in \Delta \), let \( \alpha \) be a unique positive
solution of (\ref{alpha:system}).
Assume that  \( m = \alpha n (1 + \mu n^{-1/3}) \).
Suppose that Condition \nameref{condition:C} is satisfied and \( G_{n,m,\Delta}
\) is a random graph from \( \mathcal G_{n,m,\Delta} \).

Then, as \( n \to \infty \), we have %uniformly for \( |\mu| = O(n^{1/12}) \):
\begin{enumerate}
\item if \( \mu \to -\infty \), \( |\mu|
= O(n^{1/12}) \), then 
\begin{equation}
    \mathbb P(G_{n,m,\Delta} \text{\emph{ has only trees and unicycles}}) = 1 -
\Theta(|\mu|^{-3})
    \enspace ;
\end{equation}
\item if \( |\mu| = O(1) \), i.e. \( \mu \) is fixed, then 
\begin{align}
   & \mathbb P(G_{n,m,\Delta} \text{\emph{ has only trees and unicycles}}) \to
\mbox{constant} \in
(0,1) 
    \enspace , \\
   &
   \mathbb P(G_{n,m,\Delta} \text{\emph{ has a complex part with total excess }} q)
\to \mbox{constant} \in
(0,1) %%%% \qquad (\mbox{for fixed } q)
,
\end{align}
and the constants are computable functions of \( \mu \); % as \( n \to \infty \);
\item if \( \mu \to +\infty \), \( |\mu|
= O(n^{1/12}) \), then   
\begin{align}
   & \mathbb P(G_{n,m,\Delta} \text{\emph{ has only trees and unicycles}})
= \Theta( e^{-\mu^3/6} \mu^{-3/4} ) 
    \enspace , \\
   &\mathbb P(G_{n,m,\Delta} \text{\emph{ has a complex part with excess }} q)
= \Theta( e^{-\mu^3/6} \mu^{3q/2 - 3/4} )
\enspace .
\end{align}
\end{enumerate}

\end{theorem}

\begin{proof}[Proof (Sketched)]
Consider a graph composed of trees, unicycles and a collection
of complex connected components. Fix the total excess of complex components \( q \).
Then, there are exactly \( (n - m + q) \) trees, because each tree has an excess \(
-1 \). 

Generating functions for each of these components are given by
\autoref{lemma:trees:unicycles}
and \autoref{lemma:symbolic}: we enumerate all possible kernels and then enumerate 
graphs that 
reduce to them under pruning and smoothing.
Let \( U(z) \) be the generating function for unrooted trees,
\( V(z) \) be the generating function for unicycles,  
\( E_j(z) \) be the generating functions for connected graphs with 
excess \(j \). We calculate the probability for each collection \( (q_1, \ldots,
q_k) \), while the total excess is \( \sum_{j=1}^k j q_j = q \).  
Accordingly, 
%to the principle of multiplication of exponential generating functions,
the probability that the process generates a graph with the described property
can be expressed as the ratio
\begin{equation}
 \dfrac{n!\cdot |\mathcal G_{n,m,\Delta}|^{-1}}{(n-m+q)!} [z^n]
 U(z)^{n-m+q} e^{V(z)} \dfrac{E_1^{q_1}(z)}{q_1!} \ldots \dfrac{E^{q_k}_k(z)}{q_k!} 
\enspace .
\end{equation}

Then we use an approximation of \( E_j(z) \) from 
\autoref{corollary:excess}, \autoref{lemma:symbolic} and apply
\autoref{corollary:contour} with \( y = \tfrac12 + 3q  \) in order to extract the
coefficients. Note that our approach is derived from the methods from \cite{giant_component},
and so some of our proofs are sketched.
\end{proof}
\subsection{Shifting the Planarity Threshold}

\begin{theorem}
\label{theorem:planarity}
Under the same conditions as in \autoref{theorem:critical}
with a number of edges \( m = \alpha n(1 + \mu n^{-1/3}) \), let \( p(\mu) \) be the
probability that \( G_{n,m,\Delta} \) is planar.         

Then, as \( n \to \infty \), we have uniformly for \( |\mu| = O(n^{1/12}) \):
\begin{enumerate}
\item \( p(\mu) = 1 - \Theta(|\mu|^{-3}) \),
as \( \mu \to -\infty \);
\item \( p(\mu) \to \mbox{constant} \in (0,1) \), as \( |\mu| = O(1) \),
and \( p(\mu) \) is computable;
\item \( p(\mu) \to 0 \), as \( \mu \to +\infty \).
\end{enumerate}
\end{theorem}

\begin{proof}
The graph is planar if and only if all the 3-cores (multigraphs)
of connected complex components are planar.
As \( |\mu| = O(n^{1/12}) \), \autoref{corollary:excess} implies that for asymptotic purposes
it is enough to consider only cubic
regular kernels among all possible planar 3-cores.
Let \( G_1(z) \) be an \textsc{egf} of connected planar cubic kernels.
The function \( G_1(z) \) is determined by the system of equations given in
\cite{critical_planarity}, and is computable. An \textsc{egf}
for sets of such components is given by \( G(z) = e^{G_1(z)} \).
We give several first terms of \( G(z) \) according to
\cite{critical_planarity}:   
\begin{equation}
G(z) = \sum_{q \geq 0} g_q \dfrac{z^{2q}}{(2q)!^2} = 
    1 + \dfrac{5}{24}z^2 + \dfrac{385}{1152}z^4 + 
    \dfrac{83933}{82944}z^6 + 
    \dfrac{35002561}{7962624}z^8 +
\ldots
\end{equation}

Thus, the number of planar cubic kernels with total excess \( q \) is given
by
\[
    (2q)! [z^{2q}] e^{G_1(z)} = (2q)! [z^{2q}] G(z) = \dfrac{g_q}{(2q)!}
    \enspace .
\]
In order to calculate \( p(\mu) \), we sum over all possible \( q \geq 0 \) and multiply
the probabilities that the 3-core is a planar cubic graph with excess \( q \)
by the conditional probability that a random graph has planar cubic kernel of excess \( q \).

The probability that \( G_{n,m,\Delta} \) is planar on condition that the excess
of the complex component is \( q \), is equal to
\begin{equation}
    \dfrac
    {n! |\mathcal G_{n,m,\Delta}|^{-1}}
    {(n-m+q)!}
    [z^n]
    U(z)^{n-m+q} e^{V(z)} \dfrac{g_q}{(2q)!} \dfrac{ (T_3(z))^{2r}}{ (1 -
T_2(z))^{3r}} 
    \enspace .
\end{equation}
We can apply \autoref{corollary:contour} and
sum over all \( q \geq 0 \) in order to obtain the result: 
\begin{equation}
    p(\mu) \sim \sqrt{2\pi} \sum_{q \geq 0} g_q t_3^{2q} A_{\Delta}(3q
+ \tfrac12, \mu) \enspace ,    
\end{equation}
where \( A_{\Delta}(3q+\tfrac12, \mu) \) and the constant \( t_3 \) are from \autoref{corollary:contour}.
The probabilities on
the borders of the transition window, i.e. \( |\mu| \to \infty \) can be obtained from the properties of the
function \( A_{\Delta}(y, \mu) \).
\end{proof}

\subsection{Statistics of the Complex Component Inside the Critical Window}
\begin{theorem}
\label{theorem:2paths}
%\( 1 \in \Delta \) and Condition \nameref{condition:C} holds.
Under the same conditions as in \autoref{theorem:critical},
suppose that \( {|\mu| = O(1)} \), \( m = \alpha n( 1 + \mu n^{-1/3}) \). %with \( \alpha \)
%from \autoref{theorem:critical}.
Then, the longest path, diameter and circumference of the complex part are of order
\( \Theta(n^{1/3}) \) in probability, i.e. for each mentioned random parameter there
exist computable (see~\autoref{lemma:2paths}) constants \( A, B > 0 \) depending on \( \Delta \) such that the corresponding
random variable \( X_n \) satisfies \( \forall \lambda > 0 \)
\begin{equation}
    \mathbb P \left(
        X_n \notin n^{1/3} (A \pm B \lambda)
    \right) = O(\lambda^{-2})
    \enspace . 
\end{equation}
\end{theorem}

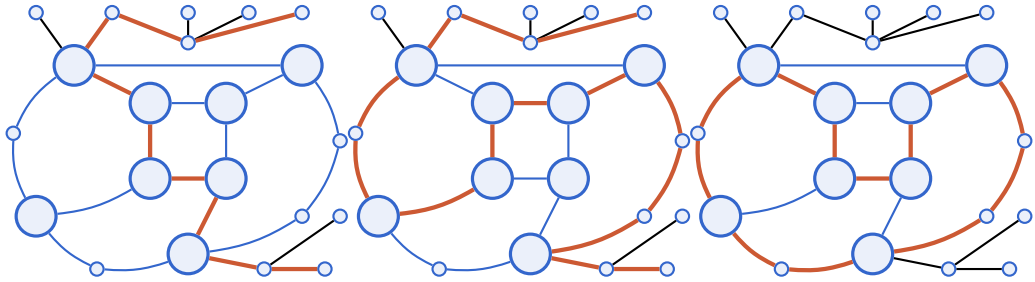
\begin{figure}[hbt]
\centering
\begin{tikzpicture}[>=stealth',thick, node distance=1.0cm] 
\draw
node[arn_n](a) at (0.0, 0.0)     {}
node[arn_nn](a1) at (-0.3, 1.1)  {}
node[arn_n](e) at (0.5, 2.0)     {}
node[arn_n](b) at (3.5, 2.0)     {}
node[arn_n](c) at (2.5, 1.5)     {}
node[arn_n](d) at (1.5, 1.5)     {}
node[arn_n](f) at (1.5, 0.5)     {}
node[arn_n](g) at (2.5, 0.5)     {}
node[arn_n](h) at (2.0, -0.5)    {}
node[arn_nn](h1) at (.8, -0.7)   {}
node[arn_nn](b1) at (4.0, 1.0)   {}
node[arn_nn](b2) at (3.5, 0.0)   {}
node[arn_nn](e1) at (0.0, 2.7)   {}
node[arn_nn](e2) at (1.0, 2.7)   {}
node[arn_nn](e3) at (2.0, 2.3)   {}
node[arn_nn](e4) at (2.0, 2.7)   {}
node[arn_nn](e5) at (2.8, 2.7)   {}
node[arn_nn](e6) at (3.5, 2.7)   {}
node[arn_nn](h2) at (3.0, -0.7)  {}
node[arn_nn](h3) at (3.8, -0.7)  {}
node[arn_nn](h4) at (4.0, 0.0)   {}
;
\path (a) edge [bblue, thick, bend left = 15] (a1) ;
\path (a1) edge [bblue, thick, bend left = 15] (e) ;
\path (e) edge [bblue, thick] (b) ;
\path (e) edge [bred, ultra thick] (d) ;
\path (d) edge [bblue,thick] (c) ;
\path (c) edge [bblue,thick] (g) ;
\path (g) edge [bred, ultra thick] (f) ;
\path (f) edge [bred, ultra thick] (d) ;
\path (a) edge [bblue, thick, bend right = 12] (f) ;
\path (a) edge [bblue, thick, bend right = 12] (h1) ;
\path (h1) edge [bblue, thick, bend right = 12] (h) ;
\path (h) edge [bred, ultra thick] (g) ;
\path (c) edge [bblue, thick] (b) ;
\path (h) edge [bblue, thick, bend right = 12] (b2) ;
\path (b2) edge [bblue, thick, bend right = 12] (b1) ;
\path (b1) edge [bblue, thick, bend right = 12] (b) ;
\path (e) edge [thick] (e1);
\path (e) edge [bred, ultra thick] (e2);
\path (e2) edge [bred, ultra thick] (e3);
\path (e3) edge [thick] (e4);
\path (e3) edge [thick](e5);
\path (e3) edge [bred, ultra thick] (e6);
\path (h) edge [bred, ultra thick] (h2);
\path (h2) edge [bred, ultra thick] (h3);
\path (h2) edge [thick] (h4);
\end{tikzpicture}%
\begin{tikzpicture}[>=stealth',thick, node distance=1.0cm] 
\draw
node[arn_n](a) at (0.0, 0.0)  {}
node[arn_nn](a1) at (-0.3, 1.1) {}
node[arn_n](e) at (0.5, 2.0) {}
node[arn_n](b) at (3.5, 2.0)  {}
node[arn_n](c) at (2.5, 1.5)  {}
node[arn_n](d) at (1.5, 1.5)  {}
node[arn_n](f) at (1.5, 0.5)  {}
node[arn_n](g) at (2.5, 0.5)  {}
node[arn_n](h) at (2.0, -0.5)  {}
node[arn_nn](h1) at (.8, -0.7)  {}
node[arn_nn](b1) at (4.0, 1.0)  {}
node[arn_nn](b2) at (3.5, 0.0)  {}
node[arn_nn](e1) at (0.0, 2.7)   {}
node[arn_nn](e2) at (1.0, 2.7)   {}
node[arn_nn](e3) at (2.0, 2.3)   {}
node[arn_nn](e4) at (2.0, 2.7)   {}
node[arn_nn](e5) at (2.8, 2.7)   {}
node[arn_nn](e6) at (3.5, 2.7)   {}
node[arn_nn](h2) at (3.0, -0.7)   {}
node[arn_nn](h3) at (3.8, -0.7)   {}
node[arn_nn](h4) at (4.0, 0.0)   {}
;
\path (a) edge [bred, ultra thick, bend left = 15] (a1) ;
\path (a1) edge [bred, ultra thick, bend left = 15] (e) ;
\path (e) edge [bblue, thick] (b) ;
\path (e) edge [bblue, thick] (d) ;
\path (d) edge [bred, ultra thick] (c) ;
\path (c) edge [bblue, thick] (g) ;
\path (g) edge [bblue, thick] (f) ;
\path (f) edge [bred, ultra thick] (d) ;
\path (a) edge [bred, ultra thick, bend right = 12] (f) ;
\path (a) edge [bblue, thick, bend right = 12] (h1) ;
\path (h1) edge [bblue, thick, bend right = 12] (h) ;
\path (h) edge [bblue, thick] (g) ;
\path (c) edge [bred, ultra thick] (b) ;
\path (h) edge [bred, ultra thick, bend right = 12] (b2) ;
\path (b2) edge [bred, ultra thick, bend right = 12] (b1) ;
\path (b1) edge [bred, ultra thick, bend right = 12] (b) ;
\path (e) edge [thick] (e1);
\path (e) edge [bred, ultra thick] (e2);
\path (e2) edge [bred, ultra thick] (e3);
\path (e3) edge [thick] (e4);
\path (e3) edge [thick](e5);
\path (e3) edge [bred, ultra thick] (e6);
\path (h) edge [bred, ultra thick] (h2);
\path (h2) edge [bred, ultra thick] (h3);
\path (h2) edge [thick] (h4);
\end{tikzpicture}%
\begin{tikzpicture}[>=stealth',thick, node distance=1.0cm] 
\draw
node[arn_n](a) at (0.0, 0.0)     {}
node[arn_nn](a1) at (-0.3, 1.1)  {}
node[arn_n](e) at (0.5, 2.0)     {}
node[arn_n](b) at (3.5, 2.0)     {}
node[arn_n](c) at (2.5, 1.5)     {}
node[arn_n](d) at (1.5, 1.5)     {}
node[arn_n](f) at (1.5, 0.5)     {}
node[arn_n](g) at (2.5, 0.5)     {}
node[arn_n](h) at (2.0, -0.5)    {}
node[arn_nn](h1) at (.8, -0.7)   {}
node[arn_nn](b1) at (4.0, 1.0)   {}
node[arn_nn](b2) at (3.5, 0.0)   {}
node[arn_nn](e1) at (0.0, 2.7)   {}
node[arn_nn](e2) at (1.0, 2.7)   {}
node[arn_nn](e3) at (2.0, 2.3)   {}
node[arn_nn](e4) at (2.0, 2.7)   {}
node[arn_nn](e5) at (2.8, 2.7)   {}
node[arn_nn](e6) at (3.5, 2.7)   {}
node[arn_nn](h2) at (3.0, -0.7)  {}
node[arn_nn](h3) at (3.8, -0.7)  {}
node[arn_nn](h4) at (4.0, 0.0)   {}
;
\path (a) edge [bred, ultra thick, bend left = 15] (a1) ;
\path (a1) edge [bred, ultra thick, bend left = 15] (e) ;
\path (e) edge [bblue, thick] (b) ;
\path (e) edge [bred, ultra thick] (d) ;
\path (d) edge [bblue, thick] (c) ;
\path (c) edge [bred, ultra thick] (g) ;
\path (g) edge [bred, ultra thick] (f) ;
\path (f) edge [bred, ultra thick] (d) ;
\path (a) edge [bblue, thick, bend right = 12] (f) ;
\path (a) edge [bred, ultra thick, bend right = 12] (h1) ;
\path (h1) edge [bred, ultra thick, bend right = 12] (h) ;
\path (h) edge [bblue, thick] (g) ;
\path (c) edge [bred, ultra thick] (b) ;
\path (h) edge [bred, ultra thick, bend right = 12] (b2) ;
\path (b2) edge [bred, ultra thick, bend right = 12] (b1) ;
\path (b1) edge [bred, ultra thick, bend right = 12] (b) ;
\path (e) edge [thick] (e1);
\path (e) edge [thick] (e2);
\path (e2) edge [thick] (e3);
\path (e3) edge [thick] (e4);
\path (e3) edge [thick](e5);
\path (e3) edge [thick] (e6);
\path (h) edge [thick] (h2);
\path (h2) edge [thick] (h3);
\path (h2) edge [thick] (h4);
\end{tikzpicture}
\caption{\label{fig:extremal:statistics}
Diameter, longest path and circumference of a complex component.
The large vertices like% 
\protect\raisebox{-2mm}{
\protect\begin{tikzpicture}[>=stealth',thick, node distance=1.0cm] 
\protect\draw
node[arn_n](a) at (0.0, 0.0)  {} ;
\protect\end{tikzpicture}} %
 are the \emph{corner} vertices
}
\end{figure}
\begin{proof}
Recall that a \emph{2-path} is a path connecting two corner vertices inside 
a complex component, see \autoref{fig:extremal:statistics}.
 In \autoref{lemma:2paths} we prove that the length of
a randomly uniformly chosen 2-path is \( \Theta(n^{1/3}) \) in probability.
This lemma also gives the explicit expressions for \( A \) and \( B \).

From \autoref{lemma:heights} we obtain that the maximum height of sprouting
tree over the complex part is also \( \Theta(n^{1/3}) \) in probability.
%if we condition on excess \( E_n = q \).
Since the total excess of the complex component is bounded in probability
as $\mu$ stays bounded, and the sizes of the kernels are finite,
we can combine these two results to
obtain the statement of the theorem, because all the three parameters come from
adding/stitching several \(2\)-paths and tree heights.

\end{proof}

\section{Simulations}
\label{section:experiments}
We considered random graphs with \( n = 1000 \) vertices, and various degree 
constraints. The random generation procedure of such graphs has been
explained by de~Panafieu and Ramos in \cite{degree_constraints} and for our experiments,
we implemented the \textit{recursive
method}. 
We note that this kind of sampling is not \textit{exact} in the sense that the
probability of obtaining a simple graph is uniform only in asymptotics.

The generator first draws a sequence of degrees and 
then performs a random pairing on half-edges, as in configuration model
\cite{configuration_model}. We reject the pairing until the 
multigraph is simple, i.e. until there are no loops and multiple edges.
As \( |\mu| = O(1) \), expected number of rejections is asymptotically
\(
	\exp\left(-\frac12\phi_1(\widehat z)-\frac14
	\phi_1^2(\widehat z)\right)
\)
, which is \( \exp(-3/4) \) in the critical window, and in the subcritical phase
it is less.

Each sequence \( (d_1, \ldots, d_n) \) is drawn with weight \( \prod_{v=1}^{n} 
1 / (d_v)! \). First, we use dynamic programming to precompute the sums of the 
weights \( (S_{i,j}) \colon {i \in [0, n]}\),
\( {j \in [0, 2m]} \)
using initial conditions and the recursive expression:
\begin{equation}
	S_{i,j} = \sum_{\overset{d_1 + \ldots + d_i = j}{d_1, \ldots, d_i \in 
	\Delta}} \prod_{v=1}^{i} \dfrac{1}{d_v!}
    \enspace 
    , \quad
	S_{i,j} = \begin{cases}
		1, & (i, j) = (0, 0) \enspace ,\\
		0, & i = 0 \text{ or } j < 0 \enspace ,\\
		\sum\limits_{d \in \Delta} \dfrac{S_{i-1, j-d}}{d!}, & \text{otherwise} 
		\enspace .
	\end{cases}
\end{equation}
Then the sequence of degrees is generated according to the distribution
\begin{equation}
	\mathbb P(d_n = d) = \dfrac{S_{n-1, 2m-d}}{d! S_{n,2m}}
	\enspace .
\end{equation}

We made plots for distributions of different parameters
for \( \Delta = \{1,3,5,7\}\), see \autoref{fig:experiments}.

\begin{figure}[hbt]
\centering
\begin{subfigure}{0.33\textwidth}
\centering
\includegraphics[width=\textwidth]{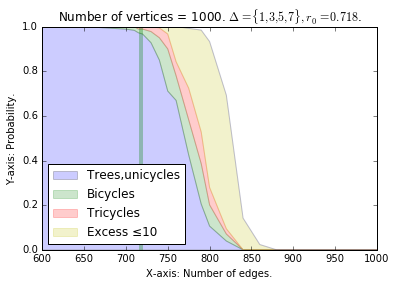}
\caption{\label{fig:largest_excess}Largest excess}
\end{subfigure}%
\begin{subfigure}{0.33\textwidth}
\centering
\includegraphics[width=\textwidth]{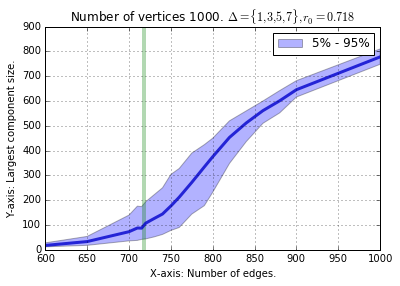}
\caption{\label{fig:largest_component}Largest component size}
\end{subfigure}%
\begin{subfigure}{0.33\textwidth}
\centering
\includegraphics[width=\textwidth]{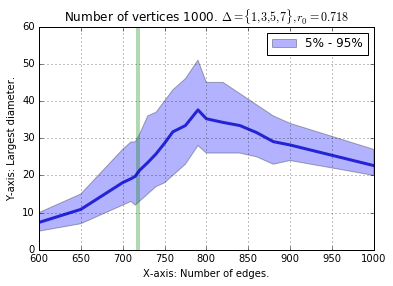}
\caption{Graph diameter}
\end{subfigure}     
\caption{\label{fig:experiments}Results of experiments}
\end{figure}

\paragraph{Conclusion.}

We studied how to shift the phase transition of
random graphs when the degrees of the nodes are
constrained by means of analytic combinatorics~\cite{degree_constraints,ac}.

We have shown that the planarity threshold of those
constrained graphs can be shifted generalizing
the results in~\cite{critical_planarity}.
We have also shown that when
our random constrained graphs are inside their critical
window of transition, the size of complex components are typically of order \(n^{2/3}\)
and all distances inside the complex components
are of order \(n^{1/3}\), thus our results about these parameters complement
those of Nachmias and Peres~\cite{mixing_time}.

A few open questions are left open: for given threshold value \( \alpha \) can
we find a set \( \{1 \} \subset \Delta \subset \mathbb Z_{\geq 0} \) delivering the
desired \( \alpha \)? What happens if \( 1 \notin \Delta \), for example what is the
structure of random Eulerian graphs? What happens when the generating function
\( \omega \) itself depends on the number of vertices?
Given a sequence of degrees \( d_1, \ldots, d_n \) that allows the construction
of a forest of an unbounded size, a first approach to study possible
relationship between the models can be the computation of the generating
function
\(
    \omega(z) = \sum_{i \geq 0} \mathrm{weight}(i) \dfrac{z^i}{i!}
\)
for a suitable weight function corresponding to \( d_1, \ldots, d_n \).
\paragraph*{Acknowledgements.}
We would like to thank Fedor Petrov for his help with \autoref{lemma:petrov}, \'{E}lie
de~Panafieu, Lutz Warnke, and several anonymous referees for their valuable remarks.
\bibliography{biblio}

\begin{thebibliography}{JPRR16}

\bibitem[BF01]{avoiding_giant_component}
Tom Bohman and Alan Freize.
\newblock Avoiding a giant component.
\newblock {\em Random Structures and Algorithms}, 19(1):75--85, 2001.

\bibitem[BLL98]{species}
Fran\c{c}ois Bergeron, Gilbert Labelle, and Pierre Leroux.
\newblock {\em Combinatorial Species and Tree-like Structures}.
\newblock Cambridge University Press, 1998.

\bibitem[Bol80]{configuration_model}
B\'{e}la Bollob\'{a}s.
\newblock A probabilistic proof of an asymptotic formula for the number of
  labelled regular graphs.
\newblock {\em European Journal of Combinatorics}, 1:311--316, 1980.

\bibitem[Bol85]{bollobas}
B\'{e}la Bollob\'{a}s.
\newblock {\em Random graphs}.
\newblock Academic Press, Inc., London, 1985.

\bibitem[dPR16]{degree_constraints}
{\'{E}}lie de~Panafieu and Lander Ramos.
\newblock Enumeration of graphs with degree constraints.
\newblock {\em Proceedings of the Meeting on Analytic Algorithmics and
  Combinatorics}, 2016.

\bibitem[ER60]{erdos_renyi}
Paul Erd\H{o}s and Alfred R\'{e}nyi.
\newblock On the evolution of random graphs.
\newblock {\em A Magyar Tudom\'{a}nyos Akad\'{e}mia Matematikai Kutat\'{o}
  Int\'{e}zet\'{e}nek K\"{o}zlem\'{e}nyei}, 5:17--61, 1960.

\bibitem[FO82]{average_height}
Philippe Flajolet and Andrew~M. Odlyzko.
\newblock The average height of binary trees and other simple trees.
\newblock {\em Journal of Computer and System Sciences}, 25:171 -- 213, 1982.

\bibitem[FPK89]{first_cycles}
Philippe Flajolet, Boris Pittel, and Donald~E. Knuth.
\newblock The first cycles in an evolving graph.
\newblock {\em Discrete Mathematics}, 75:167 -- 215, 1989.

\bibitem[FS09]{ac}
Philippe Flajolet and Robert Sedgewick.
\newblock {\em Analytic Combinatorics}.
\newblock Cambridge Press, 2009.

\bibitem[HM12]{degree_sequence_giant_components}
Hamed Hatami and Michael Molloy.
\newblock The scaling window for a random graph with a given degree sequence.
\newblock {\em Random Structures and Algorithms}, 41(1):99--123, 2012.

\bibitem[JKLP93]{giant_component}
Svante Janson, Donald~E. Knuth, Tomasz \L{}uczak, and Boris Pittel.
\newblock The birth of the giant component.
\newblock {\em Random Structures and Algorithms}, 4(3):231--358, 1993.

\bibitem[JPRR16]{fixed_degree_sequence_giant_component}
Felix Joos, Guillem Perarnau, Dieter Rautenbach, and Bruce Reed.
\newblock How to determine if a random graph with a fixed degree sequence has a
  giant component.
\newblock {\em 2016 IEEE 57th Ann. Symposium on Found. of Comp. Sc. (FOCS)},
  pages 695--703, 2016.

\bibitem[MR95]{molloy}
Michael Molloy and Bruce~A. Reed.
\newblock A critical point for random graphs with a given degree sequence.
\newblock {\em Random Structures and Algorithms}, 6(2/3):161--–180, 1995.

\bibitem[NP08]{mixing_time}
Asaf Nachmias and Yuval Peres.
\newblock Critical random graphs: Diameter and mixing time.
\newblock {\em The Annals of Probability}, 36(4):1267--1286, 2008.

\bibitem[Pet16]{242106}
Fedor Petrov.
\newblock Analytic combinatorics: upper bound for sum of absolute values of two
  complex functions: $|z f'(z)| + |2 f(z) - zf'(z)| \leq 2f(|z|)$, 2016.

\bibitem[PW13]{ACSV}
Robin Pemantle and Mark~C. Wilson.
\newblock {\em Analytic Combinatorics in Several Variables}.
\newblock Cambridge Studies in Advanced Mathematics, 2013.

\bibitem[Rio12]{Riordan12}
Oliver Riordan.
\newblock The phase transition in the configuration model.
\newblock {\em Combinatorics, Probability {\&} Computing}, 21(1-2):265--299,
  2012.

\bibitem[RRN13]{critical_planarity}
Juanjo Ru{\'e}, Vlady Ravelomanana, and Marc Noy.
\newblock The probability of planarity of a random graph near the critical
  point.
\newblock {\em Disc. Math. \& Theor. Comp. Sc.}, 2013.

\bibitem[RW12]{achlioptas_transition}
Oliver Riordan and Lutz Warnke.
\newblock {A}chlioptas process phase transitions are continuous.
\newblock {\em Annals of Applied Probability}, 22(4):1450--1464, 2012.

\bibitem[RW17]{riordan2017phase}
Oliver Riordan and Lutz Warnke.
\newblock The phase transition in bounded-size {Achlioptas} processes.
\newblock {\em arXiv preprint arXiv:1704.08714}, 2017.

\end{thebibliography}

\appendix

\section{Saddle-point Analysis}

\label{section:tools}

\subsection{Symbolic Tools}

For each \( r \geq 0 \), let us define \textit{\( r \)-sprouted trees}:
rooted trees whose vertex 
degrees belong to the set \( \Delta \), except the root
(see~\autoref{figure:recursive:tree:construction}), whose degree belongs to the set
\\
\(
	\Delta - \ell = \{ \delta \geq 0 \colon \delta + \ell \in \Delta \}
\).
Their \textsc{egf} \( T_{\ell}(z) \) can be defined recursively
\begin{equation}\label{TREE_R}
	T_{\ell}(z) = z \omega^{(\ell)}(T_1(z)), \ T_1(z) = z \omega'(T_1(z))
	\enspace , \quad \ell \geq 0
    \enspace .
\end{equation}
\begin{lemma}
\label{lemma:trees:unicycles}
Let \( U(z) \) be the \textsc{egf} for 
\textit{unrooted} trees and \( V(z) \) the EGF of
\textit{unicycles} whose vertices have degrees \(\in \Delta \). Then
\begin{equation}
	U(z) = T_0(z) - \dfrac{T_1(z)^2}{2} \enspace ,
    \quad
	V(z) = \dfrac12 \left[
		\log \dfrac{1}{1 - T_2(z)} - T_2(z) - \dfrac{T_2^2(z)}{2}
	\right] \enspace ,
\end{equation}	
where \( T_0(z) \), \( T_1(z) \), and \( T_2(z) \) are by (\ref{TREE_R}). %given above.
\end{lemma}
   
\begin{figure}[hbt]
\centering
\raisebox{1.3cm}{
\(
T_0
\begin{cases}
\phantom{r}\\
\phantom{r}\\
\phantom{r}\\
\phantom{r}\\
\phantom{r}
\end{cases}
 \)}%
\hspace{-.8cm}
\raisebox{.5cm}{
\(
T_1
\begin{cases}
\phantom{r}\\
\phantom{r}
\end{cases}
 \)}%
\hspace{-.3cm}
\begin{tikzpicture}[>=stealth',thick, node distance=1.0cm] 
\draw
node[arn_n](1)   at (0.0,  0.0)  {\( \Delta \)} %  1_7     
node[triangle](T1) at (0.0,  -2.5) {} %  3_2     
node at (0.0, -2.7) { \( \Delta\! -\! 1 \)}
node[triangle](T2) at (-3.0, -2.5) {} %  3_2     
node at (-3.0, -2.7) { \( \Delta\! -\! 1 \)}
node[triangle](T3) at (3.0,  -2.5) {} %  3_2     
node at (3.0, -2.7) { \( \Delta\! -\! 1 \)}
;
\path (1) edge [bblue,  very thick] (T1) ;
\path (1) edge [bblue,  very thick] (T2) ;
\path (1) edge [bblue,  very thick] (T3) ;
\end{tikzpicture}
\caption{\label{figure:recursive:tree:construction}
Recursive construction of \( T_0(z) \): the degree of the root of each subtree
should belong to the set \( \Delta - 1 \)}
\end{figure}
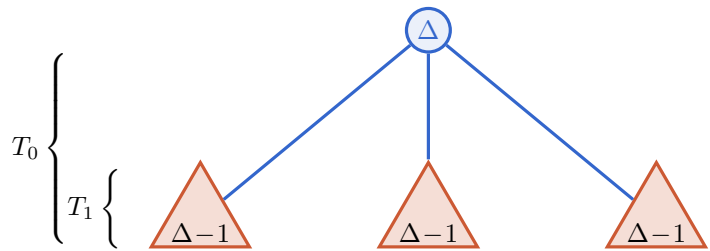

\begin{remark}
    The above statement for \( U(z) \) can be proven using the \textit{dissymetry theorem for trees}, 
adapted for the case with degree constraints (see
\cite[Section 4.1]{species}, \cite{first_cycles}).
In short, we can consider rooted trees \( T_0 \) and mark the vertex with label
\( 1 \). Then we consider two cases, when this vertex is the root one, or not.
The first case corresponds to \( U(z) \) and in the second situation we can
consider two subcases, whether the label \( 1 \) belongs to the subtree induced
by the first child or not (see~\autoref{figure:dissymetry:theorem}).
Summarizing the argument, we obtain
\[
    T_0(z) = U(z) + \dfrac12 T_1(z)^2
    \enspace .
\]

\begin{figure}[hbt]
\centering
\begin{tikzpicture}[>=stealth',thick, node distance=1.0cm, scale=0.9] 
\draw
node[btriangle](T1) at (0.0,-0.5) { \vspace{.8cm}\phantom{.}} %  3_2     
node at (0.0, -.5) { \( (\Delta) \) }
node at (1.5, -.2) { \( = \) }
node[triangle](T2) at (3.0,  0.0) {} %  3_2     
node[triangle](T3) at (5.0,  -.5) {} %  3_2     
node(T2') at (3.0,1.0) { \( \bullet \) }
node(T3') at (5.0,0.5) { \( \bullet \) }
node at (6.0, -.2) { \( + \) }
node[triangle](T4) at (7.0,  -.5) {} %  3_2     
node[triangle](T5) at (9.0, 0.0) {} %  3_2     
node at (2.6, -.2) { \( \circ 1 \) }
node(T4') at (7.0,0.5) { \( \bullet \) }
node(T5') at (9.0,1.0) { \( \bullet \) }
node at (8.6, -.2) { \( \circ 1 \) }
node at (10.3, -.2) { \( + \) }
node[vtriangle](T6) at (11.5,  -.2) {} %  3_2     
node at (11.62, 0.8) { \( \bullet 1 \) }
;
\path (T2') edge [bred,  very thick] (T3') ;
\path (T4') edge [bred,  very thick] (T5') ;
\end{tikzpicture}
\caption{\label{figure:dissymetry:theorem}
Variant of dissymetry theorem for unrooted trees with degree
constraints}
\end{figure}
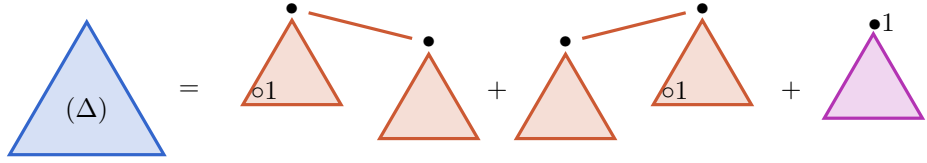

The expression for \( V(z) \)
is an application of the symbolic method of \textsc{EGF}s
in the case of undirected cycles (\textsc{cyc}$_{\geq 3}$)
 of \( 2 \)-sprouted trees,
 see~\autoref{figure:unicycles:with:degree:constraints}.

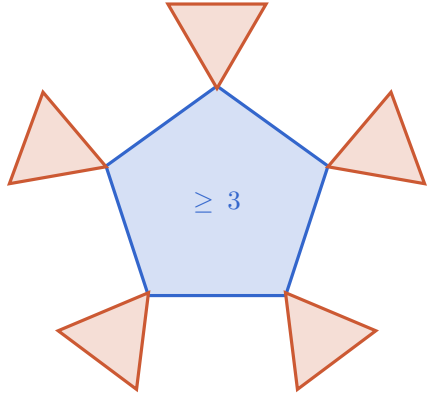
\begin{figure}[hbt]
\centering
\begin{tikzpicture}[>=stealth',thick, node distance=1.0cm, scale=0.9] 
\draw
node[bpentagon](P1) at (0.0, 0.0) { \( \geq 3 \) }
node[triangle, rotate=180](T1) at (0.0, 2.5) {}
node[triangle, rotate=180+50](T2) at (2.4, 0.8) {}
node[triangle, rotate=180-50](T4) at (-2.4, 0.8) {}
node[triangle, rotate=37](T3) at (1.5, -2.0) {}
node[triangle, rotate=-37](T5) at (-1.5, -2.0) {}
;
\end{tikzpicture}
\caption{\label{figure:unicycles:with:degree:constraints}
Unicycles with degree constraints}
\end{figure}

Any multigraph \( M \) 
on \( n \) labeled vertices can be defined by a symmetric \( n \times n \) 
matrix of nonnegative integers \( m_{xy} \), where \( m_{xy} = m_{yx} \) is the 
number of edges \( x-y \) in \( M\). The \textit{compensation factor} \( 
\kappa(M) \) is defined by
\begin{equation}
    \kappa (M) = 1 \left\slash \prod_{x=1}^n \left(
        2^{m_{xx}} \prod_{y=x}^n m_{xy}!
    \right) \right. \enspace .
\end{equation}
A \textit{multigraph process} is a sequence of \( 2m \) independent random 
vertices
\[
    (v_1, v_2, \ldots, v_{2m})
    \enspace , \quad
    v_k \in \{1, 2, \ldots, n\} 
    \enspace ,
\]
and output multigraph with the set of vertices \( \{ 1, 2, \ldots, n \} \) and 
the set of edges
\(
 \{ \{v_{2i-1}, v_{2i}\} \colon 1 \leq i \leq m \}
\).
The 
number of sequences that lead to the same multigraph \( M \) is exactly \( 2^m 
m! \kappa(M) \).
\end{remark}
\begin{lemma}
\label{lemma:symbolic}
	Let \( \overline{\overline M} \) be some 3-core multigraph with a vertex set \( V
\), \( |V| = n \), having \( \mu \) edges, and compensation factor \( 
	\kappa(\overline{\overline{M}}) 
	\). Let \( \mu_{xy} \) be the number of edges between vertices \( x \) and \( y \) 
	for \( 1 \leq x \leq y \leq n \). The generating function for all graphs \( G \) that lead to \( \overline{\overline{M}} \) 
	under reduction is
	\begin{equation}
		\dfrac{\kappa(\overline{\overline M}) \displaystyle\prod\limits_{v \in V} T_{\deg(v)} 
		(z)}{n!} \cdot 
		\dfrac{P(\overline{\overline M}, T_2(z))}{(1 - T_2(z))^{\mu}} ;
    \end{equation}
    \begin{equation}
		P(\overline{\overline M}, z) = \prod_{x = 1}^{n} \Big(
			z^{2 \mu_{xx}}\!\!\!\! \prod_{y = x+1}^{n}\!\!\!\!
			z^{\mu_{xy} - 1}
			(
			\mu_{xy} - (\mu_{xy} - 1)z
			)
		\Big).
	\end{equation}
\end{lemma}

\begin{corollary}
\label{corollary:excess}
Assume that \( \phi_1(\widehat z) = 1 \). Near the singularity \( z \sim \widehat z \), i.e. \( T_2(z) \approx 1 \),
some of the summands from \autoref{lemma:symbolic} are 
negligible.
Dominant summands correspond to graphs \( \overline{\overline{M}} \)
with maximal number of edges, i.e. graphs with \( 3r \) edges and \( 2r \) 
vertices. The vertices of degree greater than \( 3 \) 
can be splitted into more vertices with additional edges. Due to 
\cite[Section 7, Eq.~(7.2)]{giant_component}, the sum of 
the compensation factors is expressed as
\begin{equation}
	\label{eq:wright_constants}
	e_{r0} = \dfrac{(6r)!}{2^{5r}3^{2r}(3r)!(2r)!} \enspace .
\end{equation}
and the sum of major summands is asymptotically
\begin{equation}
e_{r0} \dfrac{T_3(z)^{2r}}{(1 - 
	T_2(z))^{3r}}
	\enspace .
\end{equation}
\end{corollary}

\textbf{
The proofs of the two previous statements are postponed until the end of
\autoref{remark:structural}.
}

\begin{remark}
    \label{remark:structural}
Let's give an example of application of \autoref{lemma:symbolic}, first in less
technical multigraph form, then for simple graphs.

Each multi-edge in the 3-core \( \overline{\overline{M}} \) corresponds to a
sequence of trees in the initial graph \( M \). Therefore, the generating
function for multigraphs \( M \) which reduce to one of the three depicted (see
\autoref{fig:excess:two})
3-core multigraphs consists of 3 summands:
\begin{equation}
    W_\Delta(z) = \dfrac14 \dfrac{T_4(z)}{( 1 - T_2(z))^2} +  
    \dfrac14 \dfrac{T_3(z)^2}{(1 - T_2(z))^3} + 
    \dfrac16 \dfrac{T_3(z)^2}{(1 - T_2(z))^3}
    \enspace .
\end{equation}

\begin{figure}[hbt]
\centering
\begin{tikzpicture}[>=stealth',thick, node distance=1.0cm] 
\draw
node[arn_n](a) at (0.0, 0.0) {\( 1 \)}
node at (0.0, -.7) {\(\tfrac14\)}
[bblue, very thick] (a) arc [radius = 0.5, start angle= 0, end angle= 360]
[bblue, very thick] (1.0, 0.0) arc [radius = 0.5, start angle= 0, end angle= 360]
;
\path (a) edge [bblue, very thick, bend left=90] (a) ;
\end{tikzpicture}
\( \quad \) 
\begin{tikzpicture}[>=stealth',thick, node distance=1.0cm] 
\draw
node[arn_n](a) at (0.0, 0.0) {\( 1 \)}
node[arn_n](b) at (1.0, 0.0) {\( 2 \)}
node at (0.5, -.7) {\(\tfrac14\)}
[bblue, very thick] (a) arc [radius = 0.5, start angle = 0, end angle= 360]
[bblue] (2.0, 0.0) arc [radius = 0.5, start angle = 0, end angle= 360]
;
\path (a) edge [bblue, very thick] (b) ;
\end{tikzpicture}
\( \quad \) 
\begin{tikzpicture}[>=stealth',thick, node distance=1.0cm] 
\draw
node[arn_n](a) at (0.0, 0.0) {\( 1 \)}
node[arn_n](b) at (1.4, 0.0) {\( 2 \)}
node at (0.7, -.8) {\(\tfrac16\)}
;
\path (a) edge [bblue, very thick] (b) ;
\path (a) edge [bblue, very thick, bend left = 50] (b) ;
\path (a) edge [bblue, very thick, bend right = 50] (b) ;
\end{tikzpicture}
\caption{\label{fig:excess:two}
All possible 3-core multigraphs of excess $1$ and their compensation factors.
The first one has negligible contribution because it is non-cubic
}
\end{figure}
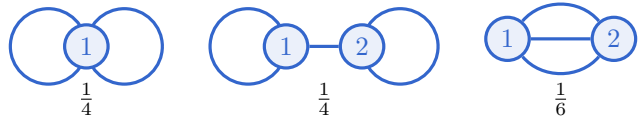

We write \( T_2(z) \) because if we attach a tree on any path, the degree of the
root decreases by $2$. For the same reason there appear \( T_3(z) \) and \(
T_4(z) \). 
If we evaluate \( W_{\Delta}(z) \) near the pole \( z = \widehat z \), or
equivalently at \( T_2(z) = 1 \), the first summand goes to \( \infty \) slower
than the second and the third. This yields asymptotic approximation%
\begin{equation}
\label{eq:obig}
    W_\Delta(z) = \dfrac{5}{24} \dfrac{T_3(z)^2}{(1 - T_2(z))^3} +
    O \left(
        \dfrac{1}{(1 - T_2(z))^2}
    \right)
    \enspace .
\end{equation}
The big-O notation with the generating functions means:
\begin{equation}
    F(z) = O(B(z)) \text{ if }
    [z^n] F(z) \leq c [z^n] B(z)
\end{equation}
for sufficiently large \( n \), so from Eq.~(\ref{eq:obig}) we know that
\begin{equation}
    [z^n] W_\Delta(z) \sim [z^n] \dfrac{5}{24} \dfrac{T_3(z)^2}{( 1- T_2(z))^3}
    \enspace .
\end{equation}

With simple graphs (not multigraphs) the situation is similar. For the first
core we want the path to be non-empty (because simple graphs don't contain
loops), so the generating function is 
\(
    \dfrac14 \dfrac{T_4(z) T_2(z)^4}{(1 - T_2(z))^2}
\).
For the second graph we also require that both paths obtained from loops,
contain at least one node inside:
\(
    \dfrac14 \dfrac{T_3(z)^2 T_2(z)^4}{(1 - T_2(z))^3}
\). Then, for the third core, we need at least two of the paths contain at least
one node. Collecting all the summands we obtain
\begin{equation}
    \widetilde W_\Delta(z) = 
    \dfrac14 \dfrac{T_4(z) T_2(z)^4}{(1 - T_2(z))^2}
    +
    \dfrac14 \dfrac{T_3(z)^2 T_2(z)^4}{(1 - T_2(z))^3}
    +
    \dfrac16 \dfrac{T_3(z)^2 [3 T_2(z)^2 - 2 T_2^3(z)]}{(1 - T_2(z))^3} 
    \enspace .
\end{equation}
At \( z \) near \( \widehat z \), \( T_2(z) = 1 \), so the asymptotics of this
term is again
\begin{equation}
    \widetilde W_\Delta (z) = \dfrac{5}{24} \dfrac{T_3(\widehat z)^2}{(1 -
T_2(z))^3} + 
    O \left(
        \dfrac{1}{(1 - T_2(z))^2}
    \right)
    \enspace .
\end{equation}
In the similar manner as was done in \cite[Lemma 2, Eq.~(9.21)]{giant_component}, we can prove that the dominant summand in the case of
simple graphs and multigraphs is the same and equals the total compensation
factor of cubic kernels \( e_{r0} \) times the generating function \(
\dfrac{T_3(z)^{\text{\#nodes}}}{(1 - T_2(z))^{\text{\#edges}}} \). We omit the
factor \( T_2(z)^{\text{\#edges}} \) because it is equal to $1$ (as \( z =
\widehat z \)).

\end{remark}
\begin{proof}[Proof of \autoref{lemma:symbolic}]
The proof is similar to \cite[Lemma 2]{giant_component}. We need to 
count the junctions of different degrees. All the paths contain vertices of 
degree at least 2, so we plug \( T_2(z) \) into \( P(\overline{\overline{M}}, z) \).
\end{proof}

\begin{proof}[Proof of \autoref{corollary:excess}](From \cite[Chapter 2]{bollobas})
In a cubic multigraph each vertex has 3 half-edges that need to be paired, there
are \( 6r \) half-edges in total. The number of such pairings is
\( (6r)! / \left( (3r)! 2^{3r} \right) \). In each vertex the three half-edges
can be permuted in \( 3! = 6 \) ways, so we divide by \( 6^{2r} \) to obtain finally
that
\begin{equation}
    \text{the number of cubic multigraphs} = \dfrac{(6r)!}{(3r)!2^{3r}6^{2r}}    
    \enspace .
\end{equation}
The multiple \( (2r)! \) appears because the graph has \( 2r \) vertices and we
deal with exponential generating functions.
\end{proof}

\subsection{Analytic Tools}

\begin{remark}
The crucial tool that we use in our work is the analytic lemma,
\autoref{lemma:contour}, or
equivalently, \autoref{corollary:contour}. Since the statements are quite cubersome, we propose
an alternative way to understand the statements, by dividing the quantities
involved into this theorem.

Suppose that \( |\mu| = O(1) \), and \( n \to \infty \). We treat \(
A_{\Delta}(y, \mu) \) as a nearly constant number, while for the asymptotics the
important factor is \( n^{y/3 - 1/2} \) with \( {y = 3r + 1/2} \).
The left-hand side of Eq.~(\ref{eq:probability}) expresses the probability
of graph having complex component of excess \( r \), provided that we specify
the function \( \Psi \) correctly according to our combinatorial specification.

When we increase the excess \( r \) by \( 1 \), the exponential index of
\( n^{y/3 - 1/2} \) increases by \( 1 \), and the
expression is multiplied by \( n \). This is exactly the combinatorial
interpretation that we are looking for: the generating functions of graphs which
reduce to kernels that are non-cubic, have a negligible contribution into the
total probability.

 In order to count the number of graphs with
complex component of excess \( r \), we note that the total number of trees
should compensate the total excess to \( m - n \), so the number of trees is \(
n - m + r \). When we substitute this into \autoref{corollary:contour}, additional multiple \(
n \) caused by extra excess, cancels with \( (n-m+r)! \) in the denominator.  
This explains why for any fixed collection of excesses of complex components \(
q_1, q_2, \ldots, q_k\) the probability of having a graph with such an excess,
is asymptotically a constant. A rigorous calculation of this probability
involves substituting the
asymptotics of \( |\mathcal G_{n,m,\Delta}| \) obtained in
\cite{degree_constraints}.
\end{remark}

\begin{lemma}
\label{lemma:contour}
Let \( m = r n = \alpha n (1 + \mu \nu) \), where
\( \nu = n^{-1/3} \), 
\( |\mu| = O(n^{1/12}) \), 
\( n \to \infty \), and
 \( \widehat z \) be a unique real positive solution of 
\( {\phi_1(\widehat z) = 1} \).
Let
\[
    C_2 = \dfrac{t_3 \alpha \widehat z}{2(1 - \alpha)} ,\quad 
    C_3 = \dfrac{2 t_3 \alpha \widehat z}{3} ,\quad
    t_3 = \dfrac{\widehat z \omega'''(\widehat z)}{\omega'(\widehat z)} .
\]
Then for any function \( \tau(z) \) analytic in
\( |z| \leq \widehat z \) the 
contour integral encircling complex zero, admits asymptotic 
representation 
\begin{equation}
\dfrac{1}{2\pi i} \oint
    (1 - \phi_1(z))^{1-y} e^{nh(z; r)} 
			\tau(z) \dfrac{dz}{z}
			\sim 
            \left.
			\nu^{2-y}(z t_3)^{1-y} \tau(z) e^{nh(z; \alpha)} \times B_{\Delta}(y, 
			\mu)\right|_{z = \widehat z} 
			\enspace ,
\end{equation}
\begin{eqnarray*}
    B_{\Delta}(y, \mu) &=& 
        \frac13 C_3^{(y-2)/3}
        \displaystyle\sum\limits_{k \geq 0}
            \dfrac
                {\left( C_2 C_3^{-2/3} \mu \right)^k}
                {k! \Gamma \left( \frac{y + 1 - 2k}{3} \right)}			
    \\
    h(z; r) &=& \log \omega' - r \log z +
    (1 - r) \log\left(2\dfrac{\omega}{\omega'} - z\right) 
\end{eqnarray*}
\end{lemma}
\textbf{The proof of the current lemma will be given below.}
\begin{remark}
In order to compute the probability in \autoref{corollary:contour}, we express the coefficient
of a generating function as a contour integral. The methods for computing
integrals of such kind are well-developed, for example, in \cite{ACSV}.
In case of single root \( z_0 \) of the derivative \( h_z(z; r) \) we approximate the integral with Gaussian density:
\begin{equation}
    \dfrac{1}{2\pi i}\oint g(z) e^{nh(z;r)} \dfrac{dt}{t} \sim
    \left( \dfrac{1}{2 \pi n} \right)^{1/2} \dfrac{g(z_0)}{z_0 \sqrt{h''(z_0)}}
    e^{nh(z_0)}
\end{equation}
and in case of double root we approximate the integral of exponential of \(
(z-z_0)^3 \).

Therefore,
\begin{equation}
    \dfrac{1}{2\pi i}\oint g(z) e^{nh(z;r)} \dfrac{dt}{t} \sim
    \dfrac{1}{2\pi i}\oint g(z) \exp\left(nh(z_0;r) +
nh^{(3)}_z(z_0;r) \dfrac{(z-z_0)^3}{6}\right) \dfrac{dt}{t}
    \enspace .
\end{equation}
Though these techniques are quite standard in a certain community, this
machinery cannot be directly applied to the case of degree constraints because
we need to prove that on the circle \(z = z_0 e^{i \theta} \), \( \theta \in [0,
2\pi] \) the maximum is attained at point \( \theta = 0 \), otherwise this
method is not applicable.
\end{remark}

\begin{corollary}
\label{corollary:contour}
	If \( m = \alpha n (1 + \mu n^{-1/3}) \) and \( y \in \mathbb R \), \( y \geq
\tfrac12 \), then for any \( \Psi(t) \) analytic at \( t=1 \) we have
\begin{equation}
\label{eq:probability}
	\dfrac{n!}{(n-m)!|\mathcal G_{n,m,\Delta}|}
    [z^n]
    \dfrac
        {U(z)^{n-m} \Psi(T_2(z))}
        {(1 - T_2(z))^{y}} 
    = 
	\sqrt{2\pi} \Psi(1)A_{\Delta}(y, \mu) n^{y/3 - 1/6} + O(R) \enspace ,
\end{equation}
\( B_{\Delta}(y,\mu) \) is from \autoref{lemma:contour} and the error term \( R \) is given by
\( 
    {R = (1 + |\mu|^{4})n^{y/3 - 1/2}}
\). This function \( A_{\Delta}(y, \mu) \) can be expressed in terms of \(
A(y,\mu) = A_{\mathbb Z_{\geq 0}}(y,\mu) \) introduced in~\cite{giant_component}:

\begin{enumerate}
\item
\(
    A_{\Delta}(y,\mu)
    =
    (t_3 \widehat z)^{1-y} (3 C_3)^{\tfrac{y-2}{3}}
    A\left( y, \dfrac{2 C_2}{ \sqrt[3]{(3 C_3)^2}} \mu \right)
    = e^{-\mu^3 / 6}
	(\widehat z t_3)^{1-y}
	B_{\Delta}(y, \mu) 
    ;
\)
\item As \( \mu \to -\infty \), we have
\(
    A(y, \mu) = \dfrac{1}{\sqrt{2 \pi} |\mu|^{y - 1/2}}
    \left(
        1 - 
        \dfrac{3 y^2 + 3y - 1}{6 |\mu|^3}
        + O(\mu^{-6})
    \right)
    ;
\)
\item As \( \mu \to +\infty \), we have
\( 
    A(y, \mu) = \dfrac{e^{-\mu^3 /6}}{2^{y/2} \mu^{1 - y/2}}
    \left(
        \dfrac{1}{\Gamma(y/2)} + 
        \dfrac{4 \mu^{-3/2}}{3 \sqrt{2} \Gamma(y/2 - 3/2)}
    + O(\mu^{-2})
    \right)
    \enspace .
\)
\end{enumerate}
\end{corollary}
                 
\begin{proof}[Proof of \autoref{lemma:contour} and \autoref{corollary:contour}]
Let us prove the corollary first.
We start with ``Stirling'' approximation part. In case of classical random
graphs it would be enough to apply the Stirling approximation, but in the case
of degree constraints we apply the asymptotic
result of \cite{degree_constraints}:
\begin{multline*}
	\dfrac{n!}{(n-m)! |\mathcal F_{n,m,\Delta}|} = \sqrt{2 \pi n} \dfrac{z_0^{2
m} \sqrt{2\alpha}}{p \cdot \omega (z_0)^{n}} \times
\\
\exp(
	n \log n + (n-m) \log(n-m) - m \log 2m
) \times \\
\exp(
	\tfrac12\phi_1(z_0) + \tfrac14 \phi_1^2(z_0)
) (1 + O(n^{-1})) \enspace .
\end{multline*}
It happens that the exponential part of Stirling and some terms that will appear
in Cauchy approximation, cancel out:
\begin{equation}
	\exp \left(n \log n + (n-m) \log(n-m)\! - m  \log 2m \right) 
	=
	e^{-\mu^3/6} \times 
	\frac{\omega(z_0)^{n}}{z_0^{2m}} 
	 2^{m-n} e^{nh(\widehat z; \alpha)} \enspace .
\end{equation}
Let us move to the Cauchy part for obtaining formal series coefficients.
After ``Lagrangian'' variable change \( T_1(z) \mapsto z \) we obtain:
\begin{align}
	\nonumber
	[z^{n}] \dfrac{U(z)^{n-m}\Psi(T_2(z))}{(1 - T_2(z))^{y}}
    = \dfrac{1}{2 \pi i} \oint 
	\dfrac{\Psi(T_2) U(z)^{n-m}
    dz}{(1 - T_2(z))^{y} z^{n+1}}
	\nonumber
	= \dfrac{2^{m-n}}{2 \pi i}
    \oint \Psi(\phi_1(z)) (1 - \phi_1(z))^{1-y} e^{nh(z; r)} 
	\dfrac{dz}{z}
    \enspace .
\end{align}
The statement readily follows from \autoref{lemma:contour}.

Let us prove the lemma then. We start with specifying an integration contour,
namely the circle \( z = \widehat z e^{-s \nu} \) where \( s = \beta + it \),
\( \beta > 0 \), \( t \in [-\pi n^{1/3}, \pi n^{1/3}] \).
We need \( \beta \to 0 \) with \( n \to \infty \).
Technically, for correct error estimate, \( \beta \) can be chosen from
\begin{equation}
    \mu = \beta^{-1} - \beta
    \enspace ,
\end{equation}
as suggested by \cite{giant_component}. We need to switch to contour \( t \in
(-\infty, +\infty) \) with the price of exponentially small error
\(
    O(e^{-\max(2, |\mu|)n^{1/6} / 3})
\),
we omit the details of this approximation since they are already considered
in the mentioned article.

Next, there will 
be two variable changes. The first change of variables is \( z = \widehat z 
e^{-s \nu} \).
We use an approximation for \( nh(z; r) \) near the double saddle \( 
\widehat z \) and critical ratio \( \alpha \). From \autoref{lemma:petrov} it follows
that maximum value of \( |e^{nh(z, r)}| \) for \( t \in [-\pi n^{1/3}, \pi
n^{1/3} ] \) is attained for \( t = 0 \) (and also
at the points \( \tfrac{2 \pi i k}{d}\nu \) where \( d \) is a period of \(
\Delta\), but we can
assume without loss of generality that \( d = 1 \), because otherwise, extra terms cancel out when we count
the probability, since the denominator is given by expression from
\cite{degree_constraints}). Thus, we can choose a small \( t_0 > 0 \), such that  \(
nh''(\widehat z)(\nu t_0)^2 \to \infty \), \( nh'''(\widehat z) (\nu t_0)^3 \to 0\),
and the absolute value of integral for \( |t| > t_0 \) is negligible.

Since there is a relation \( r =
\alpha(1 + \mu \nu) \), we can use a Taylor expansion for \( h(z, r)\) for \( z \)
around \( \widehat z\), which is uniform with
respect to \( (\alpha - r) \):
\begin{equation}
    h(z; r) = \sum_{k=0}^3 \dfrac{h^{(k)}_z (\widehat z; r) (z-\widehat
z)^k}{k!} + O\left( (s\nu)^4 \right)
    \enspace .
\end{equation}
The first derivative turns to zero, the second and the third can be written as
\begin{equation}
    h''(\widehat z, r) =
    \dfrac
        {(\phi_0(\widehat z) - 2r) \phi_1'(\widehat z)}
        {\widehat z (\phi_0(\widehat z) - 2)}
    = \dfrac{t_3 (\alpha - r) }{\widehat z (\alpha - 1) }
    \enspace ,
\end{equation}
\begin{equation}
    h'''(\widehat z, r) = 
    \dfrac
        {\phi_0'(\widehat z) \phi_1'(\widehat z)}
        {\widehat z (\alpha - 1)} + O(\mu\nu)
    \sim
        -\dfrac{4 t_3 \alpha}{\widehat z^2}
    \enspace ,
\end{equation}
hence the final approximation takes the form
\begin{equation}
    \!
	nh(z; r)\! = \! nh(\widehat z; \alpha) + C_2 \mu s^2 + C_3 s^3 +
    O\left( (\mu^2s^2 \! +\!  s^4)\nu
\right) ,
\end{equation}
where \( C_2 = h''(\widehat z; \alpha)\widehat z^2 /2 \) and \( C_3 =
-h'''(\widehat z; \alpha)\widehat z^3/6 \) are given in the formulation.
We also have
\[
    (1 - \phi_1(z))^{1-y} = s^{1-y} \nu^{1-y}(1 + O(s \nu) )
    \enspace ,
\] so
when \( s = O(n^{1/12}) \), the integrand can be approximated
\begin{multline*}
    (1 - \phi_1(z))^{1-y} e^{nh(z,r)} =
        \nu^{1-y} s^{1-y} e^{nh(\widehat z; \alpha)} \times
    e^{C_2 \mu s^2 + C_3 s^3}
    (1 + O(s \nu) + O(\mu^2 s^2 \nu) + O(s^4 \nu))
    \enspace , 
\end{multline*}
therefore 
\begin{multline*}
    \nonumber
	\dfrac{1}{2\pi i } \oint (1 - \phi_1(z))^{1-y} e^{nh(z; r)} \tau(z) 
	\dfrac{dz}{z} = 
	\dfrac{1}{2 \pi i} (\nu \widehat z \phi_1'(\widehat z))^{1-y}
	\tau(\widehat z) e^{nh(\widehat z)} \times
	\oint s^{1-y} e^{C_2 \mu s^2 + C_3 s^3} \cdot (-\nu) ds \enspace .
\end{multline*}

Then we perform a second change 
of variable \( {s = u^{1/3} C_3^{-1/3}} \). We need to be careful with the
contour of integration: note that the integral doesn't change if we take instead
of \( t \in [-\infty, + \infty] \) any
path \( \Pi(\beta), \beta > 0 \) given by
\begin{equation}
    s(t) = \begin{cases}
        -e^{-\pi i / 3} t , & - \infty < t \leq 2 \beta; \\
        \beta + it \sin \pi / 3, & -2 \beta \leq t \leq 2 \beta ; \\
        e^{ \pi i / 3} t , & 2 \beta \leq t < + \infty. 
    \end{cases}
\end{equation}
After variable transform we obtain
Hankel contour \( \Gamma \) extending from \( -\infty \), circling the
origin counterclockwise, and returning to \( -\infty \), and
\( ds = \tfrac13 (C_3 u^2)^{-1/3} du \).
\begin{multline*}
	\nonumber
	\dfrac{1}{2\pi i}\int_{\Pi(\beta)}
    (\cdots)
    = \nu^{2-y} (\widehat z \phi_1'(\widehat z))^{1-y} 
	\tau(\widehat z)e^{nh(\widehat z; \alpha)}C_3^{-y/3} 
    \times 
	\dfrac{1}{2\pi i}\int_{\Gamma} u^{y/3} e^{u} e^{C_2 \mu u^{2/3}C_3^{-2/3}} du
	\cdot
	\tfrac13 C_3^{-1/3} 
    \enspace .
\end{multline*}
Expanding the exponent
\begin{equation}
 e^{C_2C_3^{-2/3} \mu 
u^{2/3}} = \sum_{k \geq 0} (C_2C_3^{-2/3} \mu 
u^{2/3})^k/k!
\end{equation}
and applying the formula for inverse Gamma function on approximate Hankel 
contour we obtain the final statement.
\end{proof}

%%
%% Bibliography
%%

\begin{lemma}
\label{lem:hz_expression}
Assume that  \( m = rn \).
The coefficient at \( z^n \) of an \textsc{egf} for graphs from \( \mathcal
F_{n, m, \Delta} \) 
given by an equation
\begin{equation}
\label{eq:coefficient_graphs}
    [z^n] \dfrac{U(z)^{n-m+q}}{(n-m+q)!} e^{V(z)} E_{\mathbf q}(z) \enspace ,
\end{equation}
can be expressed as
\begin{equation}
    \dfrac{2^{m-n}}{2 \pi i (n-m+q)!} \oint e^{n h(z; r) } g(z) \tau(z) \dfrac{dz}{z} \enspace
, 
\end{equation}
where the contour contains \( 0 \), 
the functions \( h(z; r) \), \( g(z) \) are given by
\begin{align}
    h(z; r) = r \log \omega'(z) - r \log z + (1-r) \log (2 \omega 
     - z \omega') ,\\   
    g(z) = (1 - \phi_1(z))^{1 - y} \ , \enspace y = 3q + \tfrac12
\end{align}
and \( \tau(z) \) doesn't have singularities in \( |z| \leq \widehat z \).
\end{lemma}
\begin{proof}
We can do a variable change
\(
T_1(z) = t \mapsto z
\).\\
From the equation \( {T_1(z) = z \omega'(T_1(z))} \) we obtain:
\begin{eqnarray}
    z      &=& t \omega'(t)^{-1} \mapsto z \omega'(z)^{-1}, \\
    dz     &=& (\omega')^{-1} (1 - \phi_1) dt, \\
    T_{\ell}    &=& z \omega^{(\ell)}(t) = t \omega^{(\ell)}(t) \omega'(t)^{-1}, \\
    T_2(z) & \mapsto& \phi_1(z) \enspace ,\\
    U(z)^{n-m} &\mapsto& 2^{m-n} z^{n-m} (2 \omega (\omega')^{-1} - z)^{n-m}
\end{eqnarray}
Then, we separate out the singular part:
\begin{align}
    U(z)^{q} e^{V(z)} E_{\mathbf q}(z) \dfrac{dz}{z} =
\underbrace{
(1 - T_2(z)) (\omega')^{-1}
}_{dz / dt} \underbrace{
    \omega'(t) t^{-1}
}_{z^{-1}}
\\
\times
\underbrace{
 U(z)^q \sqrt{1 - T_2(z)} e^{V( z)}
}_{\tau_1(z)}
\underbrace{ 
(1 - T_2(z))^{3q} E_{\mathbf q}(z)
}_{\tau_2(z)}
\\
\times
\dfrac{1}{(1 - T_2(z))^{3q}} \cdot \dfrac{1}{\sqrt{1 - T_2(z)}} 
  dt 
\mapsto \tau(z) g(z) 
\enspace ,
\end{align}
and the exponential one:
\begin{equation}
    U(z)^{n-m} z^{-n} \mapsto 2^{m-n} \left(2 \dfrac{\omega}{\omega'} -
z\right)^{n-m} z^{n-m} \left(\dfrac{\omega'}{z} \right)^n         
    = 2^{m-n} e^{nh(z;r)} \enspace .
    \qquad \quad
\end{equation}
\end{proof}

%This is difficult to type zith french keyboqrd

In this section we mainly establish some asymptotic properties of \( h(z; r) \)
around \( z = \widehat z \) and \( r = r_0 = \alpha \).
Its behaviour is important for saddle-point techniques. At arbitrary point \( r = \alpha (1 + \mu n^{-1/3}) \) its derivative 
factors as
\begin{equation}
		h'_z(z; r) = \dfrac{(\phi_0(z) - 2r)(\phi_1(z) - 1)}{z (\phi_0(z) - 2)} 
		\enspace ,
\end{equation}
and the dominant complex root of \( h'_z(z;r) \) (closest to zero) is a positive 
real number which is either the solution of \( \phi_0(z) = 2r \) or the 
solution of \({ \phi_1(z) = 1 }\). Each of the equations has unique real positive 
solution which we denote by \( \mathrm{Root}_1(r) \) and \( \mathrm{Root}_2 =
\widehat z \), see~\autoref{fig:root:configuration}.

\begin{figure}[hbt]
\centering
\includegraphics[width=.8\textwidth]{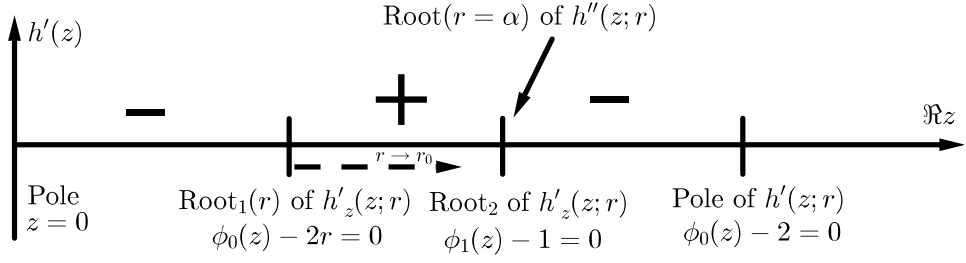}
\caption{\label{fig:root:configuration}
Configuration of roots of \( h_z(z; r) \)
}
\end{figure}

\begin{lemma}
\label{lemma:petrov}
Let \( z_0 > 0 \), \( z_0 \leq \min(\mathrm{Root}_1(r), \mathrm{Root}_2) \), 
the periodicity of \( \Delta \) is \( p \). Then the function
\begin{equation}
    \Phi(\theta; r) = \mathrm{Re}\; h(z_0 e^{i\theta}; r) 
\end{equation}
attains its global maximums for \( \theta 
\in [0, 2\pi) \) at \( p \) 
points \( \theta_k = \frac{2\pi k}{p} \), \( k = 0, 1, \ldots, p-1 \).
\end{lemma}
\begin{proof}
Denote \( z_0 e^{i\theta} \) by \( z \).
	Without loss of generality we will treat the case of aperiodic \( \omega(z) 
	\), since any \(p\)-periodic function \( \pi(z) \) can be reduced to an
aperiodic one
\( \varphi(z) \) by a variable
change \( \pi(z) = z^\ell \varphi(z^p) \).
 \( \Phi(\theta; r) \) can be rewritten as
	\begin{equation}
		\Phi(\theta; r) = r \log |\omega'(z) | + (1-r) \log |2 \omega - z \omega'| + C
        \enspace . 
	\end{equation}
	We apply a version of Gibbs inequality for Kullback-Leibler divergence, also 
	known as cross-entropy inequality: if \( p_1, p_2, q_1, q_2 \) are positive 
	real numbers and \( q_1 + q_2  \leq p_1 + p_2 \) then
	\begin{equation}
		p_1 \log \dfrac{p_1}{q_1} + p_2 \log \dfrac{p_2}{q_2} \geq 0 \enspace .
	\end{equation}
	 It is now sufficient to prove that
	\begin{equation}
		r \cdot \left|
	    	\dfrac{\omega'(z)}
	        {\omega'(z_0)}
	    \right| + 
	    (1-r) \cdot \left|
	    	\dfrac{2 \omega(z) - z\omega'(z)}
	        {2 \omega(z_0) - z_0 \omega'(z_0)}
	    \right| \leq 1 \enspace .
	\end{equation}
 Note that \( \phi_0(z_0) \leq 2 r \).
We first prove the inequality for \( \widetilde r = \tfrac12 \phi_0(z_0) \). 
% We will fix \( z_0 \) and decrease \( r \), then prove the result when \( 
% \phi_0(z_0) = 2r \).
Since the function \( \omega'(z) \) has non-negative 
 coefficients, we always  have \( |\omega'(z) / \omega'(z_0)| \leq 1 \), 
 therefore if \( r \) increases,  the inequality still remains true, thus for
all \( r \geq \widetilde r \) it is also true.
 
 Substituting \( \phi_0(z_0) = z_0 \omega'(z_0) \omega(z_0)^{-1} = 2r \) we 
 arrive to more simple inequality
\begin{equation}
\label{eq:overflow}
	|z \omega'(z)| + |2 \omega(z) - z \omega'(z)| \leq 2 \omega(z_0) \enspace , 
	\quad
    z_0 \leq \alpha \enspace .
\end{equation}
\textbf{
This inequality was proven by Fedor Petrov at mathoverflow 
\cite{242106} using a beautiful geometric statement. 
}

Let \( \gamma > \beta > 0 \) and \( 1/\beta - 1/\gamma \geq 2 \), then for any vector \( z \) with \( |z| = 1 \)
\begin{equation}
	\label{eq:geometric_lemma}
	|1 + \gamma z | + |1 - \beta z| \leq 2 + \gamma - \beta \enspace .
\end{equation}

Let us denote \( z = e^{it} \). Differentiating the expression by \( \theta \) and finding the zeros, we obtain
\begin{equation}
\dfrac{-2\gamma \sin\theta}{|1 + \gamma z|}
+ \dfrac{2 \beta \sin \theta}{|1 + \beta z|} = 0 \enspace ,
\end{equation}
which is equivalent to
\begin{equation}
	|z + 1 / \gamma| = |z - 1/\beta| \enspace ,
\end{equation}
but the middle point of the segment \( [-1/\gamma, 1/\beta] \) has value greater
than or equal to \( 1 \) provided that \( 1 / \beta - 1 / \gamma \geq 2 \), so the perpendicular bisector to this segment doesn't contain non-real points. The geometric statement in now proven.

Let \( \omega(z) = \sum_{k \geq 0} c_k z^k \).
Since \( \phi_1(\widehat z) = 1 \) and \( 0 < |z| \leq \widehat z \), the
inequality \( \phi_1(z) \leq 1 \) can be expanded as
\begin{equation}
    c_1 \geq \sum_{k \geq 2} (k^2 - 1) c_{k+1} |z|^k
    \enspace ,
\end{equation}
and we need to prove (\ref{eq:overflow}), which is equivalent to
\begin{equation}
    \left|
        \sum_{k \geq 1} kc_kz^k
    \right| +
    |2 c_0 + c_1z - c_3 z^3 - 2c_4 z^4 - \ldots|
    \leq 2 c_0 + 2 c_1 |z| + 2 c_2 |z|^2 + \ldots
\end{equation}
This is reduced by applying triangle inequality for removing terms with \( c_0 \)
and \( c_2 \) and dividing by \( |z| \):
\begin{equation}
    |c_1 + 3c_3z^3 + \ldots| + |c_1 - c_3z^2 - 2c_4z^3 - \ldots|
    \leq
    2 c_1 + 2 c_3 |z|^2 + \ldots
\end{equation}
Repeatedly using triangle inequalities, the above can be reduced to a family of
inequalities
\begin{equation}
    |(k^2 - 1)|z|^k + (k+1)z^k| + |(k^2 - 1)|z|^k - (k-1)z^k|
    \leq (2 (k^2 - 1) + 2) |z|^k \enspace ,
\end{equation}
which is a partial case of the geometric statement with \( \gamma =
\tfrac{1}{k-1} \), \( \beta = \tfrac{1}{k+1} \).
\end{proof}

\section{Method of Moments}
\label{section:marking}

In order to study the parameters of random structures, we apply the marking
procedure introduced in~\cite{ac}. We say that the
variable \( u \) marks the parameter of random structure in bivariate
\textsc{egf}
\(
    F(z, u)
\) if \( n![z^n u^k] F(z,u) \) is equal to number of structures of size \( n \)
and parameter equal to \( k \). In this section we consider
such parameters of a random graph as the length of \( 2 \)-path,
which corresponds to some edge of the \(3 \)-core, and the height of random
``sprouting'' tree.
If we treat the parameter as a random variable \( X_n \) then
the factorial moments can be
calculated through an expression
\begin{equation}
    \label{eq:factorial:moments}
    \mathbb E X_n (X_n - 1) \ldots (X_n - k + 1) = 
    \dfrac
        {\left. \dfrac{d^k}{du^k} F(z, u) \right|_{u = 1} }
        {F(z, 1)}
    \enspace .
\end{equation}

Recall that the number of graphs having \( n
\) vertices, \( m \) edges, and fixed excess vector
\( \mathbf q = (q_1, q_2, \ldots) \), can be expressed as \( n \)-th coefficient
of the generating function
\begin{equation}
    \dfrac{U(z)^{n-m+q}}{(n-m+q)!} e^{V(z)} E_{\mathbf q}(z) \enspace ,
\end{equation}
where \( E_{\mathbf q}(z) =
    \prod_{j=1}^k
    \frac{(E_j(z))^{q_j}}{q_j!}
\), \( {q = \sum_{j =1}^k jq_j} \). This \textsc{egf} can be modified to count the
moments of random variable \( X_n \).

\subsection{Length of a Random 2-path}
Let us fix the excess vector \( \mathbf q = (q_1, q_2, \ldots, q_k) \). There
are in total \( q = q_1 + 2 q_2 + \ldots + k q_k \) connected complex components and each
component has one of the finite possible number of 3-cores (see
\cite{giant_component}). We can choose any \(2\)-path,
which is a sequence of trees, and replace it with of sequence of marked trees,
see~\autoref{figure:marked:path}.
Let random variable \( P_n \) be the length of this \( 2 \)-path. Since an
\textsc{egf} for sequence of trees is \( \tfrac{1}{1 - T_2(z)} \), the
corresponding moment-generating function \( \mathbb E[u^{P_n}] \) becomes
\begin{equation}
    \mathbb E [ u^{P_n} ] =
	\dfrac
	{n! [z^n]\dfrac{U(z)^{n-m+q}}{(n-m+q)!}e^{V(z)}E_{\mathbf q}(z)
        \dfrac{1 - T_2(z)}{1 - uT_2(z)}}
	{n! [z^n]\dfrac{U(z)^{n-m+q}}{(n-m+q)!}e^{V(z)}E_{\mathbf q}(z) }
    \enspace .
\end{equation}

\begin{figure}[hbt]
\centering
\begin{tikzpicture}[>=stealth',thick, node distance=1.0cm] 
\draw node[arn_n](A1) at (0.0, 0.0){};  % {\( A_1 \)};
\draw node[arn_n](A2) at (3.0, 0.0){};  % {\( A_2 \)};
\path (A1) edge [bblue, very thick] (A2) ;
\path (A1) edge [bblue, very thick, bend left=60] (A2) ;
\path (A1) edge [bred,  very thick, bend right=60] (A2) ;
\end{tikzpicture}%
\caption{\label{figure:marked:path}
Marked 2-path inside complex component of some graph}
\end{figure}
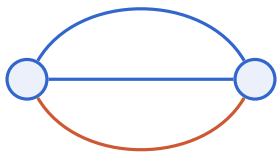

\begin{lemma}
\label{lemma:2paths}
Suppose that conditions of \autoref{theorem:critical}
are satisfied.
Suppose that there are \( q_j \) connected components of excess \( j \) for each \( j
\) from \( 1 \) to \( k \).
Denote by \emph{excess vector} a vector \( \mathbf q = (q_1, q_2, \ldots, q_k) \). 
Inside the critical window \( m = \alpha n (1 + 
\mu n^{-1/3}) \), \( |\mu| = O(1) \),
the length \( P_n \) of a random (uniformly chosen) \( 2\)-path
is \( \Theta(n^{1/3}) \) in
probability, i.e.
\begin{equation}
    \mathbb P
    \left(
        P_n \notin n^{1/3}
    t_3
    (B_1 \pm \lambda B_2) 
    \right)
    \leq
    \dfrac{1}{(\lambda + o(1))^2}
    \enspace ,
\end{equation}
\[
t_3 = \widehat z 
    \dfrac
    {\omega'''(\widehat z)}
    {\omega'(\widehat z)}
    \enspace , \quad
 B_1 =
    \dfrac
    {B_{\Delta}(3q + \tfrac32, \mu)}
    {B_{\Delta}(3q + \tfrac12, \mu)}
    \enspace ,
\]    
\[ B_2^2 =
    \dfrac
    {
        B_{\Delta}(3q + \tfrac52, \mu)
        B_{\Delta}(3q + \tfrac12, \mu) 
        -
        B_{\Delta}^2(3q + \tfrac32, \mu) 
    }
    {B_{\Delta}^2(3q + \tfrac12, \mu)}
    ,
\]    
with function \( B_{\Delta}(y, \mu) \) from \autoref{lemma:contour},
\( q = q_1 + 2 q_2 + \ldots + k q_k \).
\end{lemma}

\begin{proof}

The statement of the lemma is just an application of Chebyshev's inequality to
the first and the second moment.
Essentially, we need to prove that
\begin{equation}
    \mathbb E P_n \sim n^{1/3} 
    t_3
    \dfrac
    {B_{\Delta}(3q + \tfrac32, \mu)}
    {B_{\Delta}(3q + \tfrac12, \mu)}
    \ , \enspace
    \mathbb E P_n(P_n - 1) \sim n^{2/3} \,
    \, 2 \, t_3^2
    \dfrac
    {B_{\Delta}(3q + \tfrac52, \mu)}
    {B_{\Delta}(3q + \tfrac12, \mu)}
    \enspace ,
\end{equation}
which is just a consequence of \autoref{lemma:contour} and Eq.~(\ref{eq:factorial:moments}).
\end{proof}

\subsection{Height of a Random Sprouting Tree}

Let \( \varkappa(z) = \omega'(z) \).
Consider recursive definition for the generating function of simple trees whose
height doesn't exceed \( h \):
\begin{equation}
    \label{eq:height:definition}
    T^{[h+1]}(z) = z \varkappa(T^{[h]}(z)) \ , \enspace
    T^{[0]}(z) = 0 \enspace .
\end{equation}
The framework of multivariate generating functions allows to mark height with a
separate variable \( u \) so that the function
\begin{equation}
    F(z, u) = \sum_{n = 0}^\infty \dfrac{z^n}{n!} \sum_{h = 0}^n A^{[h]}_{n} u^h
\end{equation}
is the \textsc{bgf} for trees, where \( A^{[h]}_{n} \) stands for the number of
simple labelled rooted trees with \( n \) vertices, whose height equals \( h \).

Flajolet and Odlyzko~\cite{average_height} consider the following
expressions: 
\begin{equation}
    H(z) = \left. \dfrac{d}{du} F(z, u) \right|_{u = 1} , \enspace
    D_s(z) = \left. \dfrac{d^s}{du^s} F(z, u) \right|_{u = 1} .
\end{equation}
Generally speaking, \( H(z) = D_1(z) \) is a particular case of \( D_s(z) \),
but their analytic behaviour is different for \( s = 1 \) and \( s \geq 2 \).
\begin{lemma}[{\cite[pp. 42--50]{average_height}}]
\label{lemma:moment}
The functions \( H(z) \) and \( D_s(z) \), \( s \geq 2 \) satisfy
\begin{equation}
    H(z) \sim \alpha \log \varepsilon(z) \ , \enspace 
    D_s(z) \sim (\widehat z)^{-1} s \Gamma(s) \zeta(s) \varepsilon^{-s+1}(z)
    \enspace ,
\end{equation}
\( \alpha = 2  \dfrac{\varkappa'(\widehat z)}{\varkappa''(\widehat z)}  \), 
\( \varepsilon(z) = \widehat z \left( 1 - \dfrac{z}{\rho} \right)^{1/2}
\left( \dfrac{2 \varkappa''(\widehat z)}{\varkappa(\widehat z)} \right)^{1/2}\),
\(
    \rho
    = \widehat z \varkappa^{-1}(\widehat z)
    = (\varkappa'(\widehat z))^{-1}
\).

Here, \( \Gamma(s) \) is a gamma-function, and \( \zeta(s) \) is Riemann
zeta-funciton.
\end{lemma}                         
We don't represent their proof here, but would like to remark that it has great
methodological impact. For our purposes we need the asymptotic equivalence
\( \sim \) only in the circle of analiticity \( |z| < \rho \).

Recall that
\begin{equation}
    T_1(z) = z \omega'(T_1(z))
    \enspace , \quad
    T_\ell(z) = z \omega^{(\ell)}(T_1(z))
    \enspace , \quad
    \ell \geq 0
\end{equation}
From local expansion at \( z = \rho \) of \( z = z(T_1) \) it is easy to show that
\begin{equation}
    z \sim \rho - (T_1(z) - \widehat z)^2 \left(
        \dfrac{\varkappa''(\widehat z)\widehat z}{2 \varkappa^2(\widehat z)}
    \right)
\end{equation}
and consequently, since \( T_k(z) = z \varkappa^{(k)}(T_1(z)) \),
\begin{eqnarray}
    T_1(z) &=& \widehat z -
    \sqrt{
        \dfrac
        {2 \varkappa}
        {\varkappa ''}
    }
    \sqrt{1 - \dfrac z \rho}
    + O(1 - z \rho^{-1})
    \ , \enspace 
    \\
    T_2(z) &=& 1 - \sqrt{
        \dfrac
        {2 \widehat z \varkappa''}
        {\varkappa}
    }
    \sqrt{1 - \dfrac z \rho}
    + O(1 - z \rho^{-1})
    \enspace .
\end{eqnarray}
So we have \( \varepsilon(z) \sim \widehat z^{1/2} (1 - T_2(z)) \).

Actually, there are two kinds of sprouting trees that we have to distinguish:
the first ones are attached to the vertices with degree from \( \Delta - 2 \),
and the second~--- to the
vertices with degree from \( \Delta - 3 \), we will treat these cases
separately.

Now we can introduce 
random variables \( H_{n(2)}, H_{n(3)} \) equal to the height of a randomly uniformly chosen
sprouting tree (of the first and second type respectively), conditioned on excess number
\( \mathbf q = (q_1, q_2, \ldots, q_k) \), 
and their moment generating functions:
\begin{eqnarray}
    \mathbb E[H_{n(1)}]&=&
	\dfrac
	{[z^n]U(z)^{n-m+q}e^{V(z)}E_{\mathbf q}(z)
        \dfrac{F_{(2)}(z, u)}{T_2(z)} }
	{[z^n]U(z)^{n-m+q}e^{V(z)}E_{\mathbf q}(z) }
    \ , \enspace \\
    \mathbb E[H_{n(2)}]&=&
	\dfrac
	{[z^n]U(z)^{n-m+q}e^{V(z)}E_{\mathbf q}(z)
        \dfrac{F_{(3)}(z, u)}{T_2(z)} }
	{[z^n]U(z)^{n-m+q}e^{V(z)}E_{\mathbf q}(z) }
    \enspace ,
\end{eqnarray}
where \( F_{(2)}(z,u) \) and \( F_{(3)}(z,u) \) are the corresponding
\textsc{bgf} for \(2\)- and \(3\)-sprouted trees. 

\begin{lemma}
\label{lemma:ratio:sprouted}
Around \( z = \rho \) the derivatives of \( F_{(2)} \) and \( F_{(3)} \) with respect to \( u
\) at \( u = 1 \) can be expressed as
\begin{eqnarray*}
    \left.
    \dfrac{d^s}{du^s}F_{(2)} (z,u) \right|_{u = 1}
    &\underset{z = \rho}\sim&
    \dfrac{\varkappa'(\widehat z)}
    {\varkappa''(\widehat z)}
    \left.                      
    \dfrac{d^s}{du^s}F (z,u) \right|_{u = 1}
    \enspace , \\ 
    \left.
    \dfrac{d^s}{du^s}F_{(3)} (z,u) \right|_{u = 1}
    &\underset{z = \rho}\sim&
    \dfrac{\varkappa'(\widehat z)}
    {\varkappa'''(\widehat z)}
    \left.
    \dfrac{d^s}{du^s}F (z,u) \right|_{u = 1}
    \enspace .
\end{eqnarray*}
\end{lemma}
\begin{proof}
    We only present the main idea of the proof, omitting the technical details
of how the error term is treated~--- we refer to \cite{average_height} for the
details of transfer theorems and sum approximations.

Consider more general specification, where root degree can belong to the set \(
\Phi \) whose \textsc{egf} is given by \( \varphi(z) = \sum_{d \in \Phi} (d!)^{-1} \).
As said before, let \( T^{[h]}(z) \) be an \textsc{egf} for trees of height \(
\leq h \) given by Eq.~(\ref{eq:height:definition}). Then the
\textsc{egf} \( T_{\Phi}^{[h]}(z) \) for rooted trees, whose root belongs to \( \Phi \) with height
bounded by \( h \), can be written as
\begin{equation}
    T_{\Phi}^{[h+1]} = z \varphi(T_1^{[h]}(z)) \ , \enspace
    T_{\Phi}^{[0]}(z) = 0 \enspace .
\end{equation}                       
Then, there is a second-order Taylor expansion
\begin{multline*}
    T_{\Phi}(z) - T_{\Phi}^{[h+1]}(z) = 
    z (T_1(z) - T_1^{[h]}(z)) \varphi'(T_1(z)) \times \\
    \left[
        1 - (T_1 - T_1^{[h]}) \dfrac{\varphi''(T_1)}{2 \varphi'(T_1)} + 
        O\left( (T_1 - T_1^{[h]})^2 \right)
    \right]
    \enspace .
\end{multline*}
Denoting \( T_1 - T_1^{[h]} = e_h(z) \), 
 \( T_\Phi - T_\Phi^{[h]} = \widetilde e_h(z) \), 
we get approximate expansions
\begin{eqnarray}
    F(z,u) &\sim& u T_1(z) + (u-1) z \sum_{h \geq 1} u^h e_h(z) \varkappa'(T_1) 
\enspace ,
\\
F_\Phi(z,u) &\sim& u \varphi(T_1(z)) + (u-1) z \sum_{h \geq 1} u^h e_h(z) \varphi'(T_1) 
\enspace ,
\end{eqnarray}
so in order to calculate the ratio of derivatives with respect to \( u \) at the
vicinity of \( z = \rho \) we note that the terms \( \varkappa'(\widehat z) \) and \(
\varphi'(\widehat z) \) provide the ratio of the coefficients of main
asymptotics.
\end{proof}

\begin{lemma}
\label{lemma:heights}
Inside the critical window \( m = \alpha n(1 + \mu n^{-1/3}) \),
\( |\mu| = O(1) \),
the maximal height \( H_n \) of a sprouting tree, is of \( O(n^{1/3}) \)
in probability, i.e. 
\begin{equation}
    \mathbb P\left(
        \max H_n > \lambda n^{1/3}
    \right)  = 
    O(\lambda^{-2})
    \enspace .
\end{equation}
\end{lemma}

Actually, the \textit{average} height of a sprouting tree (if the tree is taken
uniformly at random) appears to be
\( \Theta(\log n) \) (which seems to be a new result), but when we take the maximum over all possible
\( \Theta(n^{1/3}) \)
trees, and apply Chebyshev inequality, this factor disappears.

\begin{proof}[Proof of \autoref{lemma:heights}]
    We prove the statement for 2-sprouting trees (with root degree from \(
\Delta - 2 \)), and for 3-sprouting trees the proof is the same up to a constant
term.
    
    The ratio of the expressions in the numerator and denominator can be treated
in terms of \autoref{lemma:contour}. After ``Lagrangian'' variable change \(
T_1(z) = t \mapsto z \) the ratio in \( \mathbb E H_{n(1)} \) becomes proportional to
\begin{equation}
    \dfrac
        {C_1 \displaystyle\oint (1 - \phi_1(z))^{1-y} e^{nh(z;r)} \log (1 -
\phi_1(z))dz/z}
        {C_2 \displaystyle\oint (1 - \phi_1(z))^{1-y} e^{nh(z;r)}dz/z}
\end{equation}
with \( y = 3q + \frac12 \),
and after the second variable change \( z = \widehat z e^{-s\nu} \),
%\( s = \alpha + i t \) the main asymptotics term will become
\( s = a + i t \) the main asymptotics term will become
\begin{equation}
    \left.
    \left(
        C_1 \displaystyle\oint(\cdots)
    \right)
    \right/
    \left(
        C_2 \displaystyle\oint(\cdots)
    \right)
    \sim
    \widetilde C_1(\mu) \log n
    \enspace ,
\end{equation}
For the second factorial moment we obtain
\begin{equation}
    \widetilde C_2(\mu) n^{1/3}
    + O( 1 + |\mu|^4)
    \enspace ,
\end{equation}
so from Chebyshev inequality:
\begin{equation}
    \mathbb P \left(
        |H_{n(1)} - \widetilde C_1 \log n| \geq \lambda C_2 n^{1/6}
    \right)
    \leq \dfrac{1}{(\lambda + o(1))^2}
    \enspace .
\end{equation}
Since 2-path length is \( \Theta(n^{1/3}) \) in probability, we can control the maximal tree height:
\begin{equation}
    \mathbb P( H_n \geq \lambda C_2 n^{1/3})
    = O(\lambda^{-2} n^{-1/3})
    \  , \enspace
    \mathbb P( \max H_n \geq \lambda C_2 n^{1/3}) 
    = O(\lambda^{-2})
    \enspace .
\end{equation}
\end{proof}

\end{document}